\documentclass[12 pt]{amsart}

\usepackage[margin=2.5cm]{geometry}
\usepackage{amsfonts, enumerate,mathtools}
\usepackage{amssymb}
\usepackage{amsmath}
\usepackage{latexsym}
\usepackage{mathrsfs}
\usepackage{amsthm}
\usepackage{verbatim}
\usepackage{amscd}
\usepackage{hyperref}

\usepackage{comment}
\usepackage{framed}
\usepackage{color}
\definecolor{shadecolor}{gray}{0.875}

\newtheorem{thrm}{Theorem}[section]

\newtheorem{lem}[thrm]{Lemma}
\newtheorem{cor}[thrm]{Corollary}
\newtheorem{prop}[thrm]{Proposition}
\newtheorem{conj}[thrm]{Conjecture}

\newtheorem*{conjm}{Movable Strong Conjecture (MSC)}
\newtheorem*{conjmm}{Movable Weak Conjecture (MWC)}

\theoremstyle{definition}
\newtheorem{defn}[thrm]{Definition}

\newtheorem{exmple}[thrm]{Example}

\newtheorem{rmk}[thrm]{Remark}
\newtheorem{ques}[thrm]{Question}

\newtheorem{constr}[thrm]{Construction}

\newtheorem{conv}[thrm]{Convention}

\newtheorem{claim}[thrm]{Claim}

\DeclareMathOperator{\Sym}{Sym}

\DeclareMathOperator{\Eff}{\overline{Eff}}

\DeclareMathOperator{\Nef}{Nef}

\DeclareMathOperator{\Mov}{\overline{Mov}}

\DeclareMathOperator{\mob}{mob}

\DeclareMathOperator{\bpf}{BPF}

\DeclareMathOperator{\contr}{contr}

\input xy
\xyoption{all}

\begin{document}

\title{Morphisms and faces of pseudo-effective cones}

\author{Mihai Fulger}
\address{Department of Mathematics, Princeton University\\
Princeton, NJ \, \, 08544}
\address{Institute of Mathematics of the Romanian Academy, P. O. Box 1-764, RO-014700,
Bucharest, Romania}
\email{afulger@princeton.edu}

\author{Brian Lehmann}
\thanks{The second author is supported by NSF Award 1004363.}
\address{Department of Mathematics, Boston College  \\
Chestnut Hill, MA \, \, 02467}
\email{lehmannb@bc.edu}

\begin{abstract}
Let $\pi: X \to Y$ be a morphism of projective varieties and suppose that $\alpha$ is a pseudo-effective numerical cycle class satisfying $\pi_{*}\alpha=0$. A conjecture of Debarre, Jiang, and Voisin predicts that $\alpha$ is a limit of classes of effective cycles contracted by $\pi$.  We establish new cases of the conjecture for higher codimension cycles. In particular we prove a strong version when $X$ is a fourfold and $\pi$ has relative dimension one.
\end{abstract}


\maketitle



\section{Introduction}

Let $\pi: X \to Y$ be a morphism of projective varieties over an algebraically closed field.  The pushforward of cycles induces a map $\pi_{*}: N_{k}(X) \to N_{k}(Y)$ on the groups of $\mathbb{R}$-cycles modulo numerical equivalence, and one would like to understand how $\ker(\pi_{*})$ reflects the geometry of the map $\pi$.  In the special case when $\alpha \in N_{k}(X)$ is the class of a closed subvariety $Z$, then $\alpha$ lies in the kernel of $\pi_{*}$ precisely when $\dim \pi(Z) < \dim(Z)$. A similar statement holds when $\alpha$ is the class of an effective cycle. However, the geometry of arbitrary elements of $\ker(\pi_{*})$ is more subtle.

An important idea of \cite{djv13} is that this geometric interpretation of elements of $\ker(\pi_{*})$ should be extended beyond the effective classes.  Recall that the pseudo-effective cone $\Eff_{k}(X)$ is the closure of the cone in $N_{k}(X)$ generated by the classes of effective $k$-cycles. The following is the numerical analogue of the homological statement in \cite{djv13}.

\begin{conj}\label{conj:strongweakconj}
Let $\pi: X \to Y$ be a morphism of projective varieties over an algebraically closed field.  Suppose that $\alpha \in \Eff_{k}(X)$ satisfies $\pi_{*}\alpha = 0$. Then
\begin{itemize}
\item[] \textbf{Weak Conjecture:} $\alpha$ is in the vector space generated by $k$-dimensional subvarieties that are contracted by $\pi$. 
\item[] \textbf{Strong Conjecture:} $\alpha$ is in the closure of the cone generated by $k$-dimensional subvarieties that are contracted by $\pi$.
\end{itemize}
\end{conj}

\noindent The improvement to \emph{pseudo}-effective classes is crucial for understanding the geometry of $\pi$.  For example, the interplay between morphisms from $X$ and faces of the Mori cone $\overline{\rm NE}(X)=\Eff_1(X)$ is an essential tool in birational geometry.  The Strong Conjecture predicts that for higher dimensional cycles there is still a distinguished way of constructing a face of $\Eff_{k}(X)$ from a morphism $\pi$, allowing us to deduce geometric facts from intersection theory.  The first cases of the conjecture were settled by \cite[Theorem 1.4]{djv13} which proves them for divisor classes and curve classes over $\mathbb{C}$.




The Strong Conjecture for $\pi: X \to Y$ predicts that effective classes are dense in $\ker\pi_*\cap\Eff_{k}(X)$.  It is then not surprising that the Strong Conjecture has close ties with other well-known problems predicting the existence of special cycles, as in the following example.

\begin{exmple}
Let $S$ be a smooth complex surface with $q = p_{g} = 0$.  Let $\Delta$ denote the diagonal on $S \times S$ and let $F_{1}$ and $F_{2}$ be the fibers of the projections $\pi_{1}$ and $\pi_{2}$.  Bloch's conjecture predicts that the diagonal $\Delta  - F_{1}$ is $\mathbb Q$-\emph{rationally} equivalent to a sum of cycles that are contracted by $\pi_{2}$.  In contrast, the Weak Conjecture (applied to the morphism $\pi_{2}$) predicts that $\Delta - F_{1}$ is \emph{numerically} equivalent to a sum of cycles that are contracted by $\pi_{2}$.  In this case the Weak Conjecture for $\pi_{2}$ can be verified by Hodge Theory so that $S$ admits a ``numerical'' diagonal decomposition.  

More generally, suppose that $X$ is a smooth complex variety satisfying $H^{i,0}(X) = 0$ for $i>0$. Then the Strong Conjecture for currents has implications for the Generalized Hodge Conjecture on $X \times X$ when applied to the projection maps.
This is discussed in more detail in \cite[\S6]{djv13}.\qed
\end{exmple}

As our main result, we prove the Strong Conjecture for arbitrary classes when $X$ is a fourfold and $\pi$ has relative dimension one. (This is a special case of the more general results described below.) The proof for surface classes involves new concepts and techniques concerning the positivity of higher (co)dimension cycles.  The basic principle underlying our work is that it is best to consider the Strong and Weak conjectures separately for ``movable'' classes and ``rigid'' classes.  This is motivated by the proof of the divisor case in \cite{djv13}, which relies on the $\sigma$-decomposition of \cite{nakayama04} in a fundamental way.

We first discuss the Strong Conjecture for movable classes.  In \cite{fl14} we introduced the movable cone of $k$-cycles $\Mov_{k}(X)$ which is the closure in $N_k(X)$ of the cone generated by classes of effective cycles that deform in irreducible families which cover $X$. Since movable cycles deform to cover all of $X$, morally they should not reflect the pathologies of special fibers of $\pi$. Thus it should be easier to settle the Strong and Weak Conjectures for movable classes, and in fact, the conjectures in \S\ref{s:mc} predict stronger statements. A result in this direction is the following

\begin{thrm}[cf. \ref{cor:mscalmostexc}] \label{firstmovthrm}
Let $\pi: X \to Y$ be a morphism of projective varieties over $\mathbb{C}$ of relative dimension $e$.  Fix an ample divisor $H$ on $Y$.  Suppose that $\alpha \in \Mov_{k}(X)$ for some $k \geq e$.  If $\alpha$ satisfies
\begin{equation*}
\alpha \cdot \pi^{*}H^{k-e+1}=0,
\end{equation*}
which in particular implies that $\pi_*\alpha=0$,
then the Strong Conjecture holds for $\alpha$. In particular the Strong Conjecture holds for all movable classes when $e=1$.
\end{thrm}

For perspective, note that when $\alpha$ is a movable class satisfying $\alpha \cdot \pi^{*}H^{k-e}=0$, then $\alpha=0$ by Theorem \ref{firstexcthrm} below.  Theorem \ref{firstmovthrm} handles one additional step.  

A crucial technical step is to improve our understanding of the ``dual positive classes'' defined in \cite{fl13}.  With this improvement, the proof technique is similar to \cite{lehmann11}, where the second author proves the analogous theorem for divisors.  In fact, we prove a somewhat stronger statement, allowing us to prove many cases of the Strong Conjecture for fourfolds (see Corollary \ref{cor:scfourfolds}).

We next discuss the Strong and Weak Conjectures for ``rigid'' classes.   To capture the notion of rigidity, we use the Zariski decomposition for numerical cycle classes introduced by the authors in \cite{fl15}. A Zariski decomposition of a pseudo-effective class $\alpha$ is an expression
\begin{equation*}
\alpha = P(\alpha) + N(\alpha)
\end{equation*}
where $P(\alpha)$ is a movable class that retains all the ``positivity'' of $\alpha$ and $N(\alpha)$ is \emph{pseudo}-effective.  (This decomposition is an analogue of the $\sigma$-decomposition of \cite{nakayama04}.)  \cite[Conjecture 5.19]{fl14} predicts that any negative part $N(\alpha)$ is the pushforward of a pseudo-effective class from a proper subscheme of $X$.  By definition it is a limit of classes of effective cycles, and the difficulty is showing that such cycles are contained in a fixed proper closed subset of $X$.  We establish this conjecture in a special case:

\begin{defn}
Let $\pi: X \to Y$ be a dominant morphism of projective varieties of relative dimension $e$.  Suppose that $\alpha \in \Eff_{k}(X)$.  We say that $\alpha$ is $\pi$-\emph{exceptional} if there is an ample divisor  $H$ on $Y$ such that $\alpha \cdot \pi^{*}H^{r} = 0$ for some $r \leq k-e$.
\end{defn}

When $\alpha$ is the class of a subvariety $Z$, this definition simply means that the codimension of $\pi(Z)$ is greater than the codimension of $Z$, thus extending the familiar notion for divisors. In general, this definition identifies the classes that are forced to be  ``rigid'' by the geometry of the morphism $\pi$.  A typical example is \emph{any} \emph{pseudo}-effective class in the kernel of $\pi_*$ for a birational map $\pi$.

\begin{thrm}[cf. \ref{exceptionalclassesthrm}] \label{firstexcthrm}
Let $\pi: X \to Y$ be a dominant morphism of projective varieties.  If $\alpha$ is a $\pi$-exceptional class, then:
\begin{enumerate}
\item $\alpha = 0 + N(\alpha)$ is the unique Zariski decomposition for $\alpha$.
\item $\alpha$ is the pushforward of a pseudo-effective class from a proper subscheme of $X$.
\end{enumerate}
\end{thrm}

\noindent Condition (2) implies that the Strong or Weak Conjecture for a $\pi$-exceptional class can be concluded from a statement in lower dimensions.  For example, since the Strong Conjecture is known for complex threefolds, we immediately obtain the Strong Conjecture for exceptional classes on fourfolds over $\mathbb C$.  This inductive relationship goes both ways:

\begin{prop}[cf. \ref{reddimbir} and \ref{prop:birationalequivalenttoSC}] \label{thrm:birationalnandn-1}
The Strong (resp.~Weak) Conjecture holds for birational maps $\pi: X \to Y$ of varieties of dimension $n$ if and only if the Strong (resp.~Weak) Conjecture holds in dimension $\leq n-1$.
\end{prop}

\begin{exmple}
To illustrate our techniques, in Example \ref{ccexample} we revisit the results of \cite{cc15} and \cite{luca15} which describe the geometry of higher codimension cycles on moduli spaces of pointed curves.  These papers identify classes that lie on extremal rays of the effective cone.  Using the results above, we show that their arguments actually establish extremality in the pseudo-effective cone.   (See also \cite[Remark 2.7]{cc15}.)
\end{exmple}

\subsection{Organization}

Section \ref{numeqsection} recalls the basic properties of numerical groups and positive cones.  We explain the basic features of the Strong and Weak Conjectures in Section \ref{sec:reductionsteps}.  In particular, applying ideas from \cite{fl13}, we recover many of the results of \cite{djv13}, extending some of them over an arbitrary algebraically closed field.  We also give many examples.  Rigid classes are analyzed in Section \ref{sec:excclasses} and movable classes are analyzed in Sections \ref{s:mc} and \ref{movclasssection}.  

\subsection{Acknowledgments}

Correspondence with Dawei Chen has inspired the statements in Section \ref{sec:contractibilityindex}. We also thank Zsolt Patakfalvi for helpful conversations.

\section{Background on numerical equivalence} \label{numeqsection}

\noindent By \textit{variety} we mean a reduced, irreducible, separated scheme of finite type over an algebraically closed field of arbitrary characteristic. Unless otherwise stated, $\pi:X\to Y$ is a morphism of projective varieties over the fixed ground field.

We use standard cycle constructions, but the reader should be cautioned that we work with arbitrary singularities and with numerical (and not homological or algebraic) equivalence.  Thus it is important to give the precise definitions we need.


For a projective variety $X$, we let $N_{k}(X)_{\mathbb{Z}}$ denote the quotient of the group of $\mathbb{Z}$-$k$-cycles by the relation of \emph{numerical equivalence} as in \cite[Chapter 19]{fulton84}.  $N_{k}(X)_{\mathbb{Z}}$ is a lattice inside the \emph{numerical space}
$$N_k(X):=N_k(X)_{\mathbb Z}\otimes_{\mathbb Z}\mathbb R.$$
If $Z$ is a $k$-cycle with $\mathbb{R}$-coefficients, its class in $N_k(X)$ is denoted $[Z]$.

The \emph{numerical dual group} is the vector space $N^{k}(X)$ dual to $N_{k}(X)$.  Any weighted degree $k$ homogeneous polynomial in Chern classes of vector bundles induces an element of $N^k(X)$ by intersecting against $k$-cycle Chow classes, and $N^{k}(X)$ is spanned by such elements.  The Chern class action on Chow groups descends to intersection maps $\cap: N^{r}(X) \times N_{k}(X) \to N_{k-r}(X)$.  We also use ``$\cdot$'' to denote these intersections.



\begin{conv} For the rest of the paper, the term cycle will always refer to a cycle with $\mathbb{Z}$-coefficients, and a numerical class will always refer to a class with $\mathbb{R}$-coefficients, unless otherwise qualified.
\end{conv}

\subsection{Families of cycles}

Due to subtleties of the Chow variety in positive characteristic, we will use a naive notion of a family of cycles.

\begin{defn}  \label{familydef}
Let $X$ be a projective variety.  A family of effective $k$-cycles (always with $\mathbb{Z}$-coefficients) on $X$ consists of a variety $W$, a reduced closed subscheme $U \subset W \times X$, and an integer $a_{i} \geq 0$ for each component $U_{i}$ of $U$ such that for each component $U_{i}$ of $U$ the first projection map $p: U_{i} \to W$ is flat (projective) equidimensional dominant of relative dimension $k$.  We will only consider families of cycles where each $a_{i} > 0$.
\end{defn}

The fiber over a closed point of $W$ defines a cycle $\sum_{i} a_{i} U_{i,w}$ on $X$.  As we vary $w \in W$, the resulting cycles are algebraically equivalent. We denote the corresponding numerical class by $[p]$. A family of cycles can always be extended to a projective base by using a flattening argument (see \cite[Remark 2.13]{fl14}).

\begin{constr}[Strict transform families] \label{stricttransformconstr}
Let $X$ be a projective variety and let $p: U \to W$ be a family of effective $k$-cycles on $X$.  Suppose that $\phi: X \dashrightarrow Y$ is a birational map.  We define the strict transform family of effective $k$-cycles on $Y$ as follows.

First, modify $U$ by removing all irreducible components whose image in $X$ is contained in the locus where $\phi$ is not an isomorphism.  Then define the reduced closed subset $U'$ of $W \times Y$ by taking the strict transform of the remaining components of $U$.  Over an open subset $W^{0} \subset W$, the projection map $p': (U')^{0} \to W^{0}$ is flat equidimensional on each component of $(U')^{0}$.  Each component of $U'$ is the transform of a unique component of $U$, and we assign it the same coefficient.
\end{constr}

\subsection{Positive cones}

\begin{defn}Let $X$ be a projective variety, and let $k\geq 0$. The \emph{pseudo-effective} cone
$\Eff_k(X)$ in $N_{k}(X)$ is the closure of the cone generated by classes of irreducible subvarieties of $X$.  A class is \emph{big} when it lies in the interior of $\Eff_{k}(X)$.  We use the notation $\alpha\preceq\beta$ if $\beta-\alpha\in\Eff_k(X)$.
\end{defn}

If $\pi: Y \to X$ is a surjective map, \cite[Corollary 3.22]{fl13} proves that the induced map $\pi_{*}: \Eff_{k}(Y) \to \Eff_{k}(X)$ is surjective.



It is also useful to identify the cones of ``moving cycles''.  These cones are well-studied for divisors and curves; the set-up for arbitrary cycles was considered in \cite{fl14}.

\begin{defn}Let $X$ be a projective variety.  A family of effective $k$-cycles $p: U \to W$ is \emph{strictly movable} if every component of $U$ dominates $X$.  The cycles defined by this family are called movable cycles.  When $U$ is an irreducible variety we say that $p$ is \emph{strongly movable}.  

The closure of the cone generated by classes of cycle theoretic general fibers of strongly movable families is the \emph{movable} cone $\Mov_k(X)$. Its elements are called \emph{movable} classes.  
\end{defn}

Again, if $\pi: Y \to X$ is a surjective map, \cite[Corollary 3.12]{fl14} proves that the induced map $\pi_{*}: \Mov_{k}(Y) \to Mov_{k}(X)$ is surjective.  We will also need the following key property of movable classes from \cite{fl14}:

\begin{thrm} \label{nullmovpush}
Suppose that $\pi: X \to Y$ is a generically finite map of projective varieties.  If $\alpha \in \Mov_{k}(X)$ and and $\pi_{*}\alpha = 0$, then $\alpha = 0$.
\end{thrm}


\subsection{Dual positive cones}

\begin{defn}
The \emph{nef} cone $\Nef^{k}(X)$ is the dual cone in $N^{k}(X)$ of $\Eff_{k}(X) \subset N_{k}(X)$.
\end{defn}

While nef divisors satisfy many desirable geometric properties, nef classes of higher codimension may fail to behave as well.  The basepoint free cone is introduced in \cite{fl13} as a better analogue of the nef cone of divisors.

\begin{defn}
A basepoint free family of effective $k$-cycles on a projective variety $X$ consists of
\begin{itemize}
\item an equidimensional quasi-projective scheme $U$,
\item a (necessarily equidimensional) flat morphism $s: U \to X$,
\item and a proper morphism $p:U \to W$ of relative dimension $k$ to a quasi-projective variety $W$ such that each component of $U$ surjects onto $W$.
\end{itemize}
Note that the term ``family'' here differs from that in Definition  \ref{familydef}, since $U$ is not necessarily a subset of $W\times X$, so that the fibers $U_w$ are not necessarily cycles on $X$.

The basepoint free cone $\bpf_{k}(X) \subset N_{k}(X)$ is the closure of the cone generated by the classes $F_p:=(s|_{U_w})_*[U_w]$, where $U_w$ is the fiber of a basepoint free family $p$ as above over a general $w\in W$ .
If $X$ is smooth, we define the basepoint free cone $\bpf^{k}(X) \subset N^{k}(X)$ using the isomorphism $\cap[X]$.
\end{defn}

As demonstrated by \cite{fl13}, basepoint free classes satisfy many of the desirable properties of nef divisors, improving upon the notion of nefness in higher codimension.  

We will often use a special feature of basepoint freeness, which we outline carefully here.  Suppose $\pi: Y \to X$ is birational and $p: U \to W$ is a family of $k$-cycles admitting a flat map to $X$.  In this case we can consider $p$ both as a family of cycles and  as a basepoint free family.  In this situation, the base change family on $Y$ coincides with the strict transform family as defined earlier.  (Since every component of $U$ maps dominantly onto $X$, so also every component of $U \times_{X} Y$ dominates $Y$.  Thus both families are defined via base change.) In particular, for a general member of $p$ the numerical class of the strict transform cycle is the pullback of the class of the cycle.


\section{Preliminaries on the Strong and Weak Conjectures} \label{sec:reductionsteps}

We next study the basic features of the Strong and Weak Conjectures in the numerical setting.

\begin{defn}Let $\pi:X\to Y$ be a morphism of projective varieties. Let
\begin{itemize}
\item $\Eff_k(\pi)$ denote the closed convex cone in $N_k(X)$ generated by effective $k$-classes of $X$ contracted by $\pi$,
\item $N_k(\pi)$ denote the subspace of $N_k(X)$ generated by effective $k$-classes of $X$ contracted by $\pi$,
\end{itemize}
\end{defn}

\begin{rmk}We can rephrase our conjectures and properties of interest as follows:
\begin{enumerate}[i)]
\item Weak Conjecture: $\ker\pi_*\cap \Eff_k(X)\subseteq N_k(\pi)$.
\item Strong Conjecture: $\ker\pi_*\cap \Eff_k(X)=\Eff_k(\pi)$.
\end{enumerate}
\end{rmk}

\noindent Collectively we call the Strong and Weak Conjectures the \emph{Pushforward Conjectures} and denote them by PC.  (More precisely, to say a property holds for the PC means that it holds for both the Strong Conjecture and the Weak Conjecture.)

\begin{rmk}\label{mindim}If $\dim Y<k$, then $N_k(\pi)=N_k(X)$ and PC holds trivially.\end{rmk}

As we remarked in the introduction, the PC for curves and divisors were proved in the homological setting by \cite{djv13}.  The proof for curves works well in arbitrary characteristic and in the numerical setting.  For varieties over $\mathbb{C}$, we can also deduce the PC conjectures for divisors in the numerical setting.

There are several situations where the PC conjectures are known.  They are trivially true for varieties where the pseudo-effective and effective cones coincide, such as a toric or spherical variety.  The Weak Conjecture can easily be proved for projective bundles; more generally, we have:

\begin{thrm}[\cite{fl15}] \label{firstgkthrm}
Let $\pi: X \to Y$ be a dominant morphism of projective varieties over an uncountable algebraically closed field, with $Y$ smooth.  Suppose that every fiber $F$ (over a closed point) of $\pi$ satisfies $\dim_{\mathbb Q}(CH_{0}(F)_{\mathbb Q}) =1$.  Then the Weak Conjecture holds for $\pi$.
\end{thrm}

The Strong Conjecture is much harder to prove, even for maps with an easy geometric structure.  (For example, we do not know how to prove the PC for projective bundles of relative dimension $\geq 2$.)  The difficulties are illustrated by the following example.

\begin{exmple} \label{surfaceproductexample2}
Let $S$ be a smooth surface such that $A_{0}(S) = \mathbb{Z}$.  By the work of \cite{mumford68} and \cite{roitman72}, this implies that $p_{g} = 0$ and $\mathrm{Alb}(S)$ is trivial.  Examples include any rational surface $S$ and conjecturally any surface with $q=p_{g}=0$.

Suppose that $Y$ is another smooth surface.  There is an isomorphism
\begin{equation*}
N_{2}(S \times Y) \cong \mathbb{R} \oplus (N_{1}(S) \otimes N_{1}(Y)) \oplus \mathbb{R}.
\end{equation*}
For surfaces over $\mathbb{C}$, this follows easily from Hodge theory and the K\"unneth formula.   (In fact, this argument also works for any surface $S$ satisfying $q = p_{g} = 0$.)  Over an arbitrary algebraically closed field, this follows from \cite[Theorem 1.3]{fl15}.  Note that the Weak Conjecture holds for the second projection map and for the first projection map as well.  

We now discuss the Strong Conjecture for the map $\pi=\pi_2:S\times Y\to Y$.  We first need to understand the geometry of $\pi$-vertical surfaces.  Suppose that $Z$ is an irreducible $\pi$-vertical surface so $\pi(Z)$ is a curve $C$ on $Y$.  Let $C'$ be a normalization of $C$ and $Z'$ denote the strict transform of $Z$ on $S\times C'$. If $Z'$ does not dominate $S$, then it is the pullback of a divisor on $S$.  If it does dominate $S$, then it induces a morphism $S \to \mathrm{Jac}(C')$.  But by assumption on the Albanese map this morphism is trivial.  So after twisting by the pullback of a line bundle from $S$, the divisor $Z'$ is the pullback of a divisor on $C'$.

To prove the Strong Conjecture, it suffices to consider the case when $\alpha\in\Eff_2(X)\cap\ker\pi_*$ is extremal.

\begin{claim} The Strong Conjecture holds for an extremal class $\alpha$ if and only if the projection $\alpha^{(1)}$ of $\alpha$ onto the $N_1(S)\otimes N_1(Y)$ component of $N_2(S\times Y)$ has shape $a\otimes b$ with $a\in N_1(S)$ and $b\in N_1(Y)$. (If the SC is true, then by Lemma \ref{irraprox} we can write $\alpha$ as a limit of cycles $Z_i$, each with irreducible support that does not dominate $Y$. The above argument shows that $Z_i=a_i\otimes b_i$ or $Z_i\in\mathbb R_+F_2$, and by passing to limits $\alpha=\alpha^{(1)}=a\otimes b$, or $\alpha$ is a multiple of $F_2$ and $\alpha^{(1)}=0$. Conversely, if $\alpha=a\otimes b+cF_2$, then $c\geq 0$ because it identifies with $\pi_*\alpha$. Let $\eta$ be an arbitrary nef class in $N_{1}(S)$. Then $(a\cdot \eta)b=\pi_{2*}(\alpha\cdot\pi_1^*\eta)\in\Eff_1(Y)$ and up to signs we can assume that $a$ and $b$ are both psef. Consequently $\alpha=\alpha^{(1)}+cF_2$ is a sum of psef cycles, both contracted by $\pi$ and SC is straightforward.)
\end{claim}

By the claim, the SC holds if and only if $\Eff_2(S\times Y)\cap (N_1(S)\otimes N_1(Y)\oplus\mathbb R[F_2])$ is generated by $F_2$ and by classes $a\otimes b$, where $a\in\Eff_1(S)$ and $b\in\Eff_1(Y)$. This holds for example when either $\Eff_1(S)$ or $\Eff_1(Y)$ is simplicial, e.g.~when the Picard number of $S$ or $Y$ is at most two. (Say $\Eff_1(S)$ is simplicial, generated by a basis $a_1,\ldots,a_{\rho}$ of $N^1(S)$. Then the dual basis $a_i^*$ generates $\Nef^1(S)$. Writing $\alpha$ in the unique way as $\alpha=\sum_ia_i\otimes b_i+cF_2\in\Eff_2(S\times Y)$, we see $c\geq 0$ and $\pi_{2*}(\alpha\cdot\pi_1^*(a_i^*))=b_i\in\Eff_1(Y)$.)\qed 
\end{exmple}

\subsection{Reduction steps}

\cite{djv13} shows that the (homological) SC/WC follow from certain special cases.  We note that all the reduction steps of \cite{djv13} hold in the setting of numerical equivalence using essentially the same proofs.  The following proposition is the basic tool for making this comparison.

\begin{prop}[\cite{fl13} Corollary 3.15]
\label{numchar}Let $\pi: X \to Y$ be a morphism of projective varieties.  If $\alpha\in\Eff_k(X)$ and $h$ is an ample class on $Y$, then $\pi_*\alpha=0$ if and only if $\alpha{\cdotp} \pi^*h^k=0$.
\end{prop}


A key consequence is:

\begin{cor}[\cite{djv13} Proposition 2.1]\label{finitecase} Let $\pi:X\to Y$ be a finite morphism of projective varieties. Let $\alpha\in\Eff_k(X)$. Then $\pi_*\alpha=0$ if and only if $\alpha=0$.\end{cor}

\begin{cor}The homological PC implies the numerical PC.
\end{cor}
\begin{proof}
\cite[Lemma 2.2]{djv13} shows that a \emph{pseudo-effective} cycle class $\alpha \in H_{2k}(X,\mathbb{R})$ is $0$ if and only if $h^{k} \cdot \alpha = 0$ for some/any ample divisor class $h$.  Applying the previous corollary, this is true if and only if the corresponding numerical equivalence class is $0$.  The face (or subspace) of contracted pseudo-effective homology classes maps surjectively onto the face (or subspace) of contracted pseudo-effective numerical classes. The surjectivity is clear for the linear subspace; the statement for the face requires an argument similar to the proof of \cite[Corollary 3.22]{fl13}. The conclusion of the PC clearly descends.
\end{proof}


Just as in \cite{djv13}, we then see that to prove the PC conjectures for morphisms $\pi: X \to Y$, it suffices to consider the case when:
\begin{itemize}
\item $X$ is smooth, and
\item $\pi$ is surjective with connected fibers.
\end{itemize}
In fact, given any morphism $\pi$, we can precompose by an alteration and use a Stein factorization to reduce to this situation.  

\subsection{Counterexamples to generalizations}

We next show that versions of the PC fail if we alter the hypotheses.  



\begin{exmple}
We show that the Weak Conjecture for nef classes is false.  That is, we find a nef class $\alpha$ with $\pi_{*}\alpha = 0$ which fails to be contained in the subspace generated by cycles with $\pi$-vertical support. 

Let $(S,\theta)$ be a very general principally polarized  complex abelian surface. Put $X=S\times S$. The numerical space $N^1(X)$ has a basis given by  $\{\theta_1,\theta_2,\lambda\},$ where $\theta_1$ and $\theta_2$ are the pullbacks of $\theta$ via the two projections, and $\lambda=c_1(\mathcal P)$, where $\mathcal P$ is the Poincar\' e line bundle on $X$.  Furthermore, the product map
$\Sym^2N^1(X)\to N^2(X)$
is an isomorphism. 

By \cite[Example 4.3]{delv11}, the class $\beta=8\theta_1\theta_2+3\lambda^2$ is a nef surface class. Note that $\beta \cdot\theta_1^2=0$, so that the first projection $\pi: X \to S$ yields $\pi_{*}\beta = 0$.  However, $\beta \not \in N_{k}(\pi)$ as this space is generated by $\theta_{1}^{2}, \theta_{1} \theta_{2}, \theta_{1} \lambda$.  


 We note in passing that the Weak Conjecture for pseudo-effective classes holds for $\pi$, as can easily be seen by the explicit computations of \cite[Theorem 4.1]{delv11}.  In particular $\beta$ is nef but not pseudo-effective as explained in \cite{delv11}.  \qed 
\end{exmple}

We also show that it is essential that our cone face is coming from a morphism -- even with a very minor change to the construction of the face, the analogue of the PC is false.

\begin{exmple}[Semiample intersections are necessary]\label{nefcounterexample}
If $h$ is an ample class on $Y$, consider the semiample class $\eta=\pi^*h$ and $\varphi_{\eta}$ 
the linear form on $N_k(X)$ given by $\varphi_{\eta}(\alpha)=\alpha\cdot\eta^k$.
Via Proposition \ref{numchar}, the PC say that the pseudoeffective face $\Eff_k(X)\cap\ker\varphi_{\eta}$ 
should be determined by the effective face ${\rm Eff}_k(X)\cap\ker\varphi_{\eta}$. 

From this perspective, even the Weak Conjecture may fail if $\eta$ is only nef instead of semiample: Mumford constructs an example (cf. \cite[Example 1.5.2]{lazarsfeld04}) of a surface $X$ with a nef Cartier divisor class $\eta$ such that $\eta^2=0$, and $\eta$ has positive intersection with any curve.
In particular the nonzero $\eta\in\Eff_k(X)\cap\ker\varphi_{\eta}$ cannot be determined by the empty face ${\rm Eff}_k(X)\cap\ker\varphi_{\eta}$.
\qed
\end{exmple}

\section{Exceptional classes} \label{sec:excclasses}

As discussed in the introduction, a key idea is to study the Strong and Weak Conjectures separately for movable classes and for ``rigid'' classes.  The prototypical example of a rigid divisor is an exceptional divisor for a morphism.  In this section we define and study an analogous notion for arbitrary cycle classes.
We start by recalling a useful lemma concerning cones.

\begin{lem}\label{irraprox}Let $C$ be a closed full-dimensional salient convex cone generated by a set $\{c_i\}$ inside a finite dimensional vector space. Let $\alpha\in C$ span an extremal ray. Then there exists a subsequence $\{c_j\}$ of $\{c_i\}$ and positive real numbers $r_j$ such that $\alpha=\lim_{j\to\infty}r_jc_j$.\end{lem}

\begin{prop} \label{prop:pushforwardfrome} Let $X$ be a projective variety and let $\alpha \in \Eff_{k}(X)$ be an extremal ray.  Suppose that there is an effective Cartier divisor $E$ such that $\alpha \cdot E$ is not pseudo-effective.  Then there is some component $E_{1}$ of $E$ and a pseudo-effective class $\beta \in N_{k}(E_{1})$ such that $\alpha$ is the pushforward of $\beta$ under the inclusion map.
\end{prop}

\begin{proof} We use Lemma \ref{irraprox} to write $\alpha=\lim_{j\to\infty}\alpha_j$ where $\alpha_{j} = r_{j}[S_{j}]$ for some positive real number $r_{j}$ and some irreducible $k$-dimensional subvariety  $S_j$.

Let $\{ E_{i} \}_{i=1}^{r}$ denote the components of $E$.  Note that there is some index $i$ such that an infinite subsequence of the $S_j$ is contained in $E_i$.  If this were not the case, then the classes $E \cdot S_{j}$ would be pseudo-effective for sufficiently large $j$, hence $E \cdot \alpha$ would also be pseudo-effective.  Therefore up to passing to a subsequence and renumbering the $i$'s, we can assume that $S_j\subset E_1$ for all $j$. Define the pseudo-effective classes
$$\alpha'_j=r_j[S_j]\mbox{ in }N_k(E_{1}).$$
Then letting $\imath:E_1\to X$ denote the closed embedding, we have $\alpha_j=\imath_*\alpha'_j$. Fix an ample class $h$ on $X$.  By the projection formula, 
$$(h|_{E_1})^k\cdotp\alpha'_j=h^k\cdotp\alpha_j.$$ 
\cite[Corollary 3.15]{fl13} implies that $\alpha'_j$ is a bounded sequence in $N_k(E_1)$ so that  we can extract a convergent subsequence.   If $\alpha'$ denotes the limit, then by continuity, $\imath_*\alpha'=\alpha$.
\end{proof}

\subsection{Exceptional classes}

\begin{defn} \label{def:exceptional}
Let $\pi: X \to Y$ be a morphism of projective varieties of relative dimension $e$ and fix $k\geq e$.  We say that $\alpha \in \Eff_{k}(X)$ is an exceptional (pseudo-effective) class of $\pi$ if it satisfies the following equivalent conditions:
\begin{itemize}
\item there is some (equivalently any) ample divisor $H$ on $Y$ such that $\alpha \cdot \pi^{*}H^{k-e} = 0$, or equivalently
\item there is some (equivalently any) ample divisor $A$ on $X$ such that $\pi_*(\alpha\cdot A^e)=0$.
\end{itemize}
\end{defn}

\begin{exmple}  For effective classes, being $\pi$-exceptional is a geometric condition.  Suppose that $Z$ is a subvariety of $X$.  Then $[Z]$ is $\pi$-exceptional if and only if ${\rm reldim}(X/Y) < {\rm reldim}(Z/\pi(Z))$.  (Indeed, keeping the notation from Definition \ref{def:exceptional}, note that the condition $[Z]\cdot\pi^*H^{k-e}=0$ implies $\pi|_Z^*H^{k-e}=0$.  The injectivity of $\pi|_Z^*$ for dominant maps then says $k-e>\dim\pi(Z)$.)
\end{exmple}

\begin{exmple} \label{exceptionaldivisorexample}
Suppose $E$ is an effective divisor.  Then $[E]$ is exceptional for a morphism $\pi$ precisely when $\mathrm{codim}(\pi(E)) \geq 2$, that is, when $E$ is exceptional by the usual definition.

In fact, it follows from Theorem \ref{exceptionalclassesthrm} below that any pseudo-effective $\pi$-exceptional divisor class $\alpha$ is represented by an effective divisor that is exceptional in the traditional sense.
\end{exmple}

\begin{exmple} \label{exmple:exceptionalforbirational}
Let $\pi: X \to Y$ be birational, or only generically finite and dominant.  Then a pseudo-effective class is $\pi$-contracted if and only if it is $\pi$-exceptional.
\end{exmple}

\begin{exmple}
A pseudo-effective curve class can only be exceptional for generically finite dominant maps.
\end{exmple}

\begin{lem}\label{flatexceptional}
Suppose $\pi: X \to Y$ is an equidimensional morphism of projective varieties.  Then there are no $\pi$-exceptional classes on $X$ besides the $0$ class.
\end{lem}
\begin{proof}
Let $Z$ be a complete intersection on $X$ of codimension $e={\rm reldim}(\pi)$
with embedding morphism $i:Z\to X$. Let $f=\pi|_Z$. By Lemma \ref{lem:cutdownflatmap}, if we choose $Z$ general we may ensure that $f$ is finite. If $\alpha$ is $\pi$-exceptional, then $f_*(i^*\alpha)=0$. By Lemma \ref{psefrestriction} below, $i^*\alpha$ is pseudoeffective. By Corollary \ref{finitecase} and the projection formula, $\alpha\cdot [Z]=0$ on $X$, which implies $\alpha=0$ since $[Z]$ is in the interior of $\Nef^e(X)$.
\end{proof}

\begin{lem} \label{lem:cutdownflatmap}
Let $\pi: X \to Y$ be a map of projective varieties, with equidimensional fibers of relative dimension $\geq 1$.  Fix a very ample divisor $H$ on $X$.  For some sufficiently large integer $m$ and for a general member $A$ of $|mH|$, the fibers of $\pi: A \to Y$ are equidimensional. 
\end{lem}

\begin{proof}
For degree reasons there is an upper bound on the number of components of a fiber of $\pi$. We first show that the supports of the irreducible components of a general fiber of $\pi$ are parameterized by a quasi-projective variety $Z$ dominating $Y$ generically finitely. In particular $\dim Z=\dim Y$. For the claim, let $k$ denote the base field, and note that there exists a finite extension $K(Y)\subset K$ such that every irreducible component of $X_K$ as a scheme over ${\rm Spec}\, K$ is geometrically irreducible. Let $\mathcal T$ be the support of a component of $X_K$ that dominates the generic fiber of $\pi$. Then there exists a projective morphism $\rho:T\to Z$ over $k$ with $Z$ affine, $K(Z)=K$, with $\mathcal T$ the generical fiber of $\rho$, and $Z$ finite dominant over an affine open subset of $Y$. Furthermore $\rho$ has irreducible and reduced fibers, each a component of a general fiber of $\pi$.  

Returning to the lemma, it suffices to show that there are elements of $|mH|$ which do not contain a component of a fiber of $\pi$ over any general closed point of $Y$.  Fix a general fiber $F$ of $\pi$ and let $F_{1},\ldots,F_{s}$ be its components.  Consider the exact sequence:
\begin{equation*}
H^{0}(X,\mathcal{I}_{F_{i}} \otimes \mathcal{O}_{X}(mH)) \hookrightarrow H^{0}(X,\mathcal{O}_{X}(mH)) \to H^{0}(F_{i},\mathcal{O}_{F_{i}}(mH)) \to H^{1}(X,\mathcal{I}_{F_{i}} \otimes \mathcal{O}_{X}(mH))
\end{equation*}
For $m$ sufficiently large, the last term vanishes for every $i$ and the dimension of the next to last term is larger than $\dim\, Z=\dim\, Y$ for every $i$.  Furthermore, since $F$ is general the same statements will hold for any general fiber of $\pi$.  Thus, $\dim |mH|$ is strictly greater than the dimension of the space of divisors containing a component of $F$ plus $\dim\, Y$.  Constructing the incidence correspondence, one sees that the general element of $|mH|$ does not contain a component of any general fiber.  After removing the proper closed subset of $|mH|$ parametrizing divisors which contain a component of the special fibers, we obtain the conclusion of the theorem.
\end{proof}


\begin{lem}\label{psefrestriction}
Let $X$ be a projective scheme, and let $D$ be a complete intersection of ample divisors of codimension $d$ with embedding morphism $\imath:D\hookrightarrow X$. If $\alpha\in\Eff_k(X)$, then $\alpha|_D:=\imath^*\alpha\in\Eff_{k-d}(D)$.
\end{lem}

\begin{proof}By induction we can assume that $D$ is a divisor. By continuity and additivity of intersections, we can assume that $\alpha=[Z]$ for some irreducible subvariety $Z$ of $X$.
If $D$ and $Z$ meet properly, then $[Z]|_D$ is the class of an effective cycle supported on $D\cap|Z|$. Otherwise $Z\subset D$ and then $[Z]|_D$ is the pushforward of $c_1(\mathcal O_Z(D))\cap[Z]$ from $Z$ to $D$. But $\mathcal O_Z(D)$ is ample and the conclusion follows.
\end{proof}

The key lemma controlling the behavior of exceptional classes is the following.

\begin{lem} \label{exceptionalintersectsnonpsef}
Let $\pi: X \to Y$ be a dominant morphism of projective varieties. There is a birational model $f_{X}: X' \to X$ and an effective Cartier divisor $E$ on $X'$ satisfying the following condition: for any $k$ and for any $0\neq\alpha' \in \Eff_{k}(X')$ such that $f_{X*} \alpha'$ is a $\pi$-exceptional class, $E \cdot \alpha'$ is not pseudo-effective. Furthermore the support of $E$ does not dominate $Y$.

In the special case when $\pi$ is generically finite, we may take $X' = X$.
\end{lem}

\begin{proof}
Let $\pi': X' \to Y'$ be a flattening of $\pi$ and let $f_{X}, f_{Y}$ denote the corresponding birational maps.  Set $e={\rm reldim}(\pi)$ and $\alpha = f_{X*}\alpha'$.  As in Definition \ref{def:exceptional} the condition that $\alpha$ is $\pi$-exceptional is equivalent 
to $$f_{Y*}\pi'_*(\alpha'\cdot f_X^*A^e)=0$$ for any ample class $A$ on $X$.  Let $F$ be an effective $f_Y$-anti-ample Cartier divisor on $Y'$. Then $E=\pi'^*F$ is an effective $f_X$-anti-ample Cartier divisor on $X'$ (since $f_X$ is the composition of a finite inclusion map with a base change of $f_Y$).  

Suppose that $E\cdot\alpha'$ is pseudo-effective.  By Lemma \ref{dercycle}, there is an ample divisor $A'$ on $X'$ satisfying $\alpha' \cdot f_X^{*}A^{e} \succeq \alpha' \cdot A'^{e}$.  Thus $f_{Y*}\pi'_*(\alpha'\cdot (A')^e)=0$ so that $\alpha'$ is $f_Y\circ\pi'$-exceptional. Furthermore, since
\begin{equation*}
F\cdot\pi'_*(\alpha'\cdot (A')^e)  = \pi_{*}(E\cdot\alpha' \cdot (A')^{e})
\end{equation*}
is pseudo-effective and $f_{Y}$ is generically finite, Lemma \ref{dercycle} implies
$\pi'_*(\alpha'\cdot (A')^e)=0$, i.e.~$\alpha'$ is $\pi'$-exceptional.  But this is impossible by Lemma \ref{flatexceptional}. 

To see the final statement, note that by considering the Stein factorization of $\pi$ one immediately reduces to the birational case; but then the flattening of $\pi$ is the identity map of $X$.
\end{proof}

\begin{lem}\label{dercycle}Let $\pi:X\to Y$ be a generically finite dominant morphism of projective varieties, and let $E$ be an effective $\pi$-antiample Cartier divisor on $X$. Let $\alpha\in\Eff_k(X)$ and let $H$ be an ample divisor on $Y$ such that $\pi^*H-E$ is ample. If $\alpha\cdot [E]\in\Eff_{k-1}(X)$, then $\alpha\cdot\pi^*H^e\succeq\alpha\cdot(\pi^*H-E)^e$ for any $1\leq e\leq k$. \end{lem}
\begin{proof}Observe that $$\pi^*H^e-(\pi^*H-E)^e=E\cdot \left( \sum_{i=1}^e  \pi^*H^{i-1}(\pi^*H-E)^{e-i} \right),$$ and $\sum_{i=1}^e \pi^*H^{i-1}(\pi^*H-E)^{e-i}$ is a positive combination of complete intersections of nef divisors.\end{proof}

\begin{thrm} \label{exceptionalclassesthrm}
Let $\pi: X \to Y$ be a morphism of projective varieties of relative dimension $l$ and let $\alpha$ be a $\pi$-exceptional class.
\begin{enumerate}
\item The only way to write $\alpha = P + N$ for a movable class $P$ and a pseudo-effective class $N$ is if $P=0$ and $\alpha = N$.
\item $\alpha$ is the pushforward of a pseudo-effective class on a proper subvariety of $X$.  
\end{enumerate}
In fact, there is a proper closed subset $W \subset X$ such that \emph{every} $\pi$-exceptional class is the pushforward of a pseudo-effective class on $W$.
\end{thrm}

\begin{proof}
(1) Since $\alpha$ is $\pi$-exceptional, so are $P$ and $N$.  Choose a flattening $\pi': X' \to Y'$ of $\pi$ and let $\alpha' \in \Mov_{k}(X')$ be a preimage of $P$.  Lemma \ref{exceptionalintersectsnonpsef} implies that there is an effective Cartier divisor $E$ on $X'$ such that $\alpha' \cdot E$ is not pseudo-effective if $\alpha'\neq 0$.  This is impossible by Theorem \ref{nullmovpush}, showing that $P = 0$ as well.

(2) It suffices to consider the case when $\alpha$ lies on an extremal ray.  Choose a flattening $\pi': X' \to Y'$ of $\pi$.  Then there is some $\alpha'$ lying on an extremal ray of $\Eff_{k}(X')$ that is a preimage of $\alpha$ and an effective Cartier divisor $E$ satisfying $E \cdot \alpha'$ is not pseudo-effective. Proposition \ref{prop:pushforwardfrome} shows that $\alpha'$ is the pushforward of a pseudo-effective class from a subvariety $E_{1}$ of $X'$.  Then $\alpha$ is the pushforward of a pseudo-effective class from the image of $E_{1}$ in $X$.

To see the final statement, note that by Lemma \ref{exceptionalintersectsnonpsef} we  can choose $E_{1}$ independently of $\alpha$ in (2); set $W = f_{X}(|E_{1}|)$. 
\end{proof}

\subsection{Contractibility index} \label{sec:contractibilityindex}

It turns out to be useful to quantify ``how close'' a $\pi$-contracted class is to being $\pi$-exceptional.

\begin{defn}
Let $\pi: X \to Y$ be a dominant morphism of projective varieties.  Suppose $H$ is an ample divisor on $Y$.  For a class $\alpha \in \Eff_{k}(X)$, the contractibility index of $\alpha$ is the largest non-negative integer $c \leq k+1$ such that $\alpha \cdot \pi^{*}H^{k+1-c} = 0$.  The definition is independent of the choice of $H$.  We denote the $\pi$-contractibility index of $\alpha$ by $\contr_{\pi}(\alpha)$.


If $V \subset X$ is a subvariety, we define the contractibility index of $V$ to be the contractibility index of $[V]$.
\end{defn}

Note that when $c=0$ we have $\alpha \cdot \pi^{*}H^{k+1} = 0$ for dimension reasons, so that the $\pi$-contractibility index is well-defined.  The following properties are immediate:

\begin{itemize}
\item A pseudo-effective class has $\contr_{\pi}(\alpha) > 0$ precisely when $\pi_{*}\alpha = 0$.
\item The contractibility index is (by definition) at most $k+1$, and $0 \in \Eff_{k}(X)$ is the only pseudo-effective class achieving this maximal value.
\item The contractibility index of $\alpha \in \Eff_{k}(X)$ is at least $k - \dim Y$.
\item The contractibility index of $[X] \in N_{\dim X}(X)$ is the relative dimension of $\pi$.  More generally, if $\alpha$ is the class of an irreducible cycle $Z$, then $\contr_{\pi}(\alpha) = \mathrm{reldim}(\pi|_{Z})$.
\item A pseudo-effective class $\alpha$ is $\pi$-exceptional precisely when its contractibility index is greater than $\mathrm{reldim}(\pi)$.
\end{itemize}


The following theorem was inspired by a question of Dawei Chen.

\begin{thrm} \label{contindexextremality}
Let $\pi: X \to Y$ be a dominant morphism of projective varieties.  Fix a positive integer $m$.  Let $k=k(m)$ be the largest integer such that there is a subvariety of dimension $k$ of contractibility index $\geq m$.  Then
\begin{enumerate}
\item There are only finitely many subvarieties $V_{1},\ldots,V_{s}$ of dimension $k$ and contractibility index $\geq m$.
\item Any $\alpha \in \Eff_{k}(X)$ of contractibility index $\geq m$ is a non-negative linear combination of the $[V_{i}]$.
\end{enumerate}
\end{thrm}

\begin{proof}
Set $n=\dim X$.  The proof is by induction on the codimension $n-k$.  The base case $n-k=0$ is obvious.  

Now suppose $n-k>0$.  In particular $m$ is greater than the contractibility index of $X$, so that any class with contractibility index $\geq m$ is $\pi$-exceptional.  Theorem \ref{exceptionalclassesthrm} guarantees that there is a proper subscheme $i: W \hookrightarrow X$ such that every $\pi$-exceptional class is pushed forward from $W$.   Let $\{ W_{i} \}$ be the irreducible components of $W$.  Note that for any pseudo-effective class $\beta$ on a component $W_{i}$, the contractibility index for $\pi|_{W_{i}}$ is the same as the contractibility index of the pushforward of $\beta$ to $X$.

In particular, for each $W_{i}$ any subvariety with contractibility index $\geq m$ has dimension no more than $k$.  Applying the induction assumption to each $W_{i}$ in turn, we immediately obtain (1) for $X$.  Suppose now that $\alpha \in \Eff_{k}(X)$ has contractibility index $\geq m$.  Since $\alpha$ is $\pi$-exceptional, it is pushed forward from $W$, and hence is a sum of pushforwards of pseudo-effective classes of contractibility index $\geq m$ from the $W_{i}$.  Applying (2) to each $W_{i}$, we obtain (2) for $X$ by pushing forward.
\end{proof}

If $\pi:X\to Y$ is birational and $Z\subset X$ is the exceptional locus of $\pi$, then there are no nonzero effective $\pi$-exceptional (equivalently $\pi$-contracted by Example \ref{exmple:exceptionalforbirational}) classes of dimension bigger than $\dim Z$. As a consequence of the theorem above, the same is true for pseudo-effective classes. In particular the SC holds in this case.

\begin{cor}Let $\pi:X\to Y$ be a birational morphism and let $Z={\rm Exc}(\pi)\subset X$. If $d=\dim Z$, then $\Eff_k(\pi)=0$ for all $k>d$ and $\Eff_d(\pi)$ is polyhedral, generated by the components of $Z$ that are contracted by $\pi$. 
\end{cor}

\begin{exmple} \label{ccexample}
\cite{cc15} and \cite{luca15} identify cycle classes which lie on extremal rays of the \emph{effective} cone of various moduli spaces of curves.  
Some of these results can be extended to the \emph{pseudo}-effective setting.   We will prove one of these extensions as a demonstration of our techniques.

\cite[Proposition 2.5]{cc15} considers a generically finite morphism $\pi: X \to Y$ of projective varieties.  Suppose that $Z\subset X$ is a subvariety which contains the $\pi$-exceptional locus and the pushforward $N_{k}(Z) \to N_{k}(X)$ is injective.  Then (under some additional hypotheses) \cite{cc15} shows that any effective class $\alpha \in N_{k}(X)$ which is $\pi$-exceptional, and the pushforward of a class lying on an extremal ray of ${\rm Eff}_k(Z)$, is also extremal in $\Eff_{k}(X)$.

The analogous statement for pseudo-effectivity is also true, even without the additional hypotheses.  Suppose that $\beta_{i}$ are pseudo-effective classes on $X$ satisfying $\sum \beta_{i} = \alpha$.  Then each $\beta_{i}$ is also $\pi$-exceptional; arguing as in Lemma \ref{exceptionalintersectsnonpsef}, we see that any such $\beta_{i}$ is the pushforward of a pseudo-effective class on $N_{k}(Z)$.  Using the injectivity of the pushforward, we deduce that $\sum \beta_{i} = \alpha$ as classes on $N_{k}(Z)$ -- by extremality, each $\beta_{i}$ must be proportional to $\alpha$ in $N_{k}(Z)$, hence also in $N_{k}(X)$.)

Using these strengthened versions, one can extend many of the results of \cite{cc15} to the pseudo-effective cone.  For example, \cite[Theorem 7.2]{cc15} shows that the effective cone of codimension $2$ cycles on $\overline{M}_{1,n}$ is not finite polyhedral.  The analogous statement for $\Eff_{n-1}(\overline{M}_{1,n})$ is also true.  Indeed, one just needs to replace \cite[Proposition 2.5]{cc15} by the corresponding pseudo-effective extension above.  The same argument in the context of \cite[Theorem 8.2]{cc15} shows that the pseudo-effective cone of codimension $2$ cycles on $\overline{M}_{2,n}$ is not finite polyhedral.

\end{exmple}

\subsection{Further reduction steps}

Theorem \ref{exceptionalclassesthrm} allows us to make some further reductions to the pushforward conjectures beyond \cite{djv13}.  Note that it is crucial to allow singularities to make these further reductions.

\begin{prop}[Reducing dimension in the relative dimension zero case]\label{reddimbir} Let $\pi:X\to Y$ be a generically finite dominant morphism of projective varieties of dimension $n$. Let $E_i$ be the components of an effective $\pi$-antiample Cartier divisor $E$ of class $e$ on $X$. If the PC are true for $\pi|_{E_i}$, then they are also true for $\pi$. 
\end{prop}

\begin{proof} Let $\alpha\in\Eff_k(X)$ with $\pi_*\alpha=0$. By Example \ref{exmple:exceptionalforbirational}, $\alpha$ is $\pi$-exceptional. We may assume that $\alpha$ is also extremal. The last statement in Lemma \ref{exceptionalintersectsnonpsef} shows that $\alpha$ is the pushforward of a pseudo-effective class from a component of $E$, whence the result.
\end{proof}

In particular, to deduce the SC/WC by an inductive argument, we may always precompose by generically finite maps.

\begin{prop}[Reduction to the flat case]The general case of the PC is implied by
the flat case.\end{prop}

\begin{proof}Let $\pi:X\to Y$ be a morphism of projective varieties, and let $\alpha\in\Eff_k(X)\cap\ker\pi_*$ which we can assume to be extremal. 
Let $\pi':X'\to Y'$
be a flattening of $\pi$ with birational morphisms $f_X:X'\to X$ and $f_Y:Y'\to Y$. Let $\alpha'$ be an extremal pseudo-effective preimage of $\alpha$ on $X'$. It is enough to show that the PC hold for the pair $(\pi\circ f_X,\alpha')$. Let $F$ be an effective $f_Y$-anti-ample Cartier divisor on $Y'$, and let $H$ be an ample class
on $Y$ such that $f_Y^*H-F$ is ample on $Y'$. Put $E=\pi'^*F$.

If $\alpha'\cdot E\in\Eff_{k-1}(X')$, then by the projection formula, $\pi'_*\alpha'\cdot F\in\Eff_{k-1}(Y')$. Proposition \ref{numchar} and Lemma \ref{dercycle} then show that $\pi'_*\alpha'=0$. If the PC hold for the pair $(\pi',\alpha')$, then they also hold for $(\pi,\alpha)$. 

If $\alpha'\cdot E$ is not pseudo-effective, then Proposition \ref{prop:pushforwardfrome} shows that $\alpha'$ is pushed forward from one of the irreducible components of the support of $E$. Conclude by induction on $\dim X$. 
\end{proof}

Finally we show that the Strong or Weak Conjecture follows from the birational case. The argument which works by increasing dimension is based on a relative version of a classical cone construction. 

\begin{lem}\label{relcone}Let $\pi:X\to Y$ be a projective morphism. Let $H$ be a sufficiently
$\pi$-ample divisor on $X$, and let \begin{equation}\label{eq:conevsstein}\mathcal R(H):=\mathcal O_Y\oplus
\bigoplus_{m\geq 1}\pi_*\mathcal O_X(mH).\end{equation} Put $T={\rm Spec}_{\mathcal O_Y}\mathcal R(H)$ with induced affine morphism $\rho:T\to Y$. The natural
map $$\mathcal R(H)\twoheadrightarrow\mathcal O_Y$$ induces a section $i:Y\hookrightarrow T$ of $\rho$. Then $Z:={\rm Bl}_{i(Y)}T$ admits a morphism 
$f:Z\to X$ that is isomorphic to the bundle map for the geometric line bundle
$Z':={\rm Spec}_{\mathcal O_X}\bigoplus_{m\geq 0}\mathcal O_X(mH)$.
Moreover if $E$ denotes the exceptional divisor of the blow-up, then $f|_E:E\to X$ is an isomorphism whose inverse is the zero section of $f$, and the induced map $E\to Y$ is naturally isomorphic to $\pi$.
\end{lem}


\begin{rmk}If one replaces $\mathcal O_Y$ with $\pi_*\mathcal O_X$ in degree 0 in the formula for $\mathcal R(H)$ in \eqref{eq:conevsstein}, then $T={\rm Spec}_{\mathcal O_Y}\mathcal R(H)$ is the cone over the Stein factorization of $\pi$.
\end{rmk}

\begin{prop} \label{prop:birationalequivalenttoSC}
If SC (or WC) holds for birational maps, then it holds in general.
More precisely, if the conjecture is valid for birational maps in dimension $n+1$,
then it is valid for all morphisms $\pi:X\to Y$ of projective varieties with $\dim X=n$.
\end{prop}
\begin{proof} Let $\pi:X\to Y$ be a morphism of projective varieties, with $\dim X=n$. Consider $\eta:W\to X$ a projective bundle of rank 1 over $X$ that compactifies the geometric line bundle from Lemma \ref{relcone}. Let $E\subset W$ be the zero section of $\eta$. By Lemma \ref{relcone}, the divisor $E\subset W$ can be blown-down as $\sigma:W\to S$, with $\sigma(E)=Y$, and $\sigma|_E^Y$ is naturally isomorphic to $\pi$. The proposition is a consequence of the following identifications:
$$N_k(X)=E\cdot\eta^*N_k(X)\subset N_k(W),$$
$$N_k(\pi)=N_k(\sigma),$$
which are compatible with the respective pseudo-effective cones. 
\end{proof}


\section{The Movable Strong and Weak Conjectures} \label{s:mc}

We next focus our attention on the Strong and Weak Conjectures for movable classes.  It seems likely that whenever the Strong/Weak Conjectures hold, a stronger version will hold for movable classes.

\begin{defn}Let $\pi:X\to Y$ be a morphism of projective varieties. Let
\begin{itemize}
\item $\Mov_k(\pi)$ denote the closed convex cone in $N_k(X)$ generated by movable effective $k$-classes of $X$ contracted by $\pi$,
\item $M_k(\pi)$ denote the subspace of $N_k(X)$ generated by movable effective $k$-classes of $X$ contracted by $\pi$.
\end{itemize}
\end{defn}

The analogues of the Strong and Weak Conjectures for movable classes are:

\begin{conjm} \label{conj:msc}
Let $\pi: X \to Y$ be a morphism of projective varieties.  Let $\alpha \in \Mov_{k}(X)$ satisfy $\pi_{*}\alpha = 0$.  Then $\alpha$ belongs to $\Mov_{k}(\pi)$.
\end{conjm}

\begin{conjmm} \label{conj:mwc}
Let $\pi: X \to Y$ be a morphism of projective varieties.  Let $\alpha \in \Mov_{k}(X)$ satisfy $\pi_{*}\alpha = 0$.  Then $\alpha$ belongs to $M_{k}(\pi)$.
\end{conjmm}

There are no immediate implications between these conjectures and the Strong/Weak Conjectures.

These conjectures seem substantially easier than their non-movable counterparts.  The key insight is that movable cycles should not reflect the properties of special fibers of a map, inviting arguments that only rely on general fibers.  For example, Theorem \ref{nullmovpush} implies that the Movable Strong Conjecture is true for generically finite maps $\pi$ (compare with Corollary \ref{finitecase}).  This should be contrasted with Proposition \ref{prop:birationalequivalenttoSC}, which shows the Strong Conjecture in dimension $n$ for birational maps is equivalent to the Strong Conjecture in all lower dimensions (and for maps of arbitrary relative dimension).

We denote the Movable Strong Conjecture and Movable Weak Conjecture collectively as the Movable Pushforward Conjectures (MPC). As with the case of PC, the MPC reduce immediately to 
maps $\pi: X \to Y$ where
\begin{itemize}
\item $X$ is smooth,
\item $\pi$ is surjective, and has connected fibers
\end{itemize}

\begin{rmk}\label{bothsmconnC}Over $\mathbb C$, we can also assume that $Y$ is smooth. (Let $\pi:X\to Y$ be a surjective projective morphism with connected fibers with $X$ smooth, and let $\pi':X'\to Y'$ be a birational  model of $\pi$ with $X'$ and $Y'$ both smooth, and with birational morphisms $f_X:X'\to X$ and $f_Y:Y'\to Y$. We can assume that $\pi$ and $\pi'$ coincide over an open subset of $Y$. Consequently $\pi'$ also has connected fibers. If $\alpha'$ is a movable $f_X$-lift of $\alpha$ (as in \cite[Corollary 3.12]{fl14}), then $(f_Y)_*(\pi'_*\alpha')=0$ implies $\pi'_*\alpha'=0$ by Theorem \ref{nullmovpush}. If the MPC are true for $\pi'$ and $\alpha'$, then they are also true for $\pi$ and $\alpha$.)
\end{rmk}


\subsection{The Movable Strong Conjecture for curves}



\begin{thrm}
Let $\pi: X \to Y$ be a morphism of projective varieties over $\mathbb{C}$.  Then the MPC hold for curve classes on $X$.
\end{thrm}

\begin{proof}
Note that the MWC and MSC can be detected after precomposing by any surjective map.  So we may suppose $X$ is smooth. Then apply Theorem \ref{nullmovpush} to the normalization of the Stein factorization of $\pi$ to reduce to the case when $Y$ is normal and $\pi$ has connected fibers. This case is then settled by \cite[6.8 Theorem]{peternell12}. Specifically, 
he shows that if $F$ is a very general fiber of $\pi$,
with inclusion morphism $i:F\to X$, then $i_{*}(\Mov_{1}(F))= \Mov_{1}(X) \cap \ker(\pi_{*}).$
\end{proof}

\begin{cor} \label{cor:msccurvesreldim1}
Let $\pi: X \to Y$ be a surjective morphism of projective varieties over $\mathbb{C}$ of relative dimension $1$.  Let $\alpha$ be a movable curve class with $\pi_{*}\alpha = 0$.  Then $\alpha$ is proportional to the class of a general fiber.
\end{cor}

\begin{proof}We may assume that $X$ is smooth, $Y$ is normal, and $\pi$ has connected fibers.
Then the general fiber of $\pi$ is a smooth curve $F$ and $\Mov_1(F)$ is clearly spanned by $[F]$. Conclude by \cite[6.8 Theorem]{peternell12}.
\end{proof}


\subsection{The Movable Strong Conjecture for divisors} \label{sec:movstrongconjdiv}

\begin{thrm}
Let $\pi: X \to Y$ be a morphism of projective varieties over $\mathbb{C}$.  Then the MPC hold for divisor classes on $X$.
\end{thrm}

\begin{proof}
Let $\alpha$ be a movable divisor class on $X$ satisfying $\pi_{*}\alpha = 0$.  By Remark \ref{mindim} we reduce to the case when $c:={\rm reldim}(\pi)\in\{0,1\}$.
When $\pi$ is birational, Theorem \ref{nullmovpush} shows that $\alpha = 0$.  So the MSC holds in this case.

Suppose that $\pi$ has relative dimension $1$.  Using the reduction techniques of \cite{djv13}, we may assume that $X$ and $Y$ are nonsingular and that $\pi$ has connected fibers.  If $b$ is the class of a fiber of $\pi$, then by Corollary \ref{numchar} we have $\alpha \cdot b=0$, so the restriction of $\alpha$ to the general fiber of $\pi$ is numerically trivial. By \cite[Theorem 1.3]{lehmann11}, there exists a commutative diagram 
$$\xymatrix{X'\ar[r]^{f}\ar[d]_{\pi'}& X\ar[d]^{\pi}\\
Y'\ar[r]_g& Y}$$
with $X'$ and $Y'$ nonsingular, with $f$ and $g$ birational, and a pseudo-effective divisor class $\beta$ on $Y'$ such that $P_{\sigma}(f^*\alpha) = P_{\sigma}(\pi'^*\beta)$. It is enough to prove the proposition for $\pi'$ and $P_{\sigma}(f^*\alpha)$. Thus we may assume $\pi=\pi'$ and
$$\alpha =P_{\sigma}(\pi^*\beta)$$ for some pseudo-effective divisor class $\beta$ on $Y$.

Write $\beta$ as a limit of big classes $\beta_i:=\beta+\frac 1{i}\delta$, where $\delta$ is a big divisor class on $Y$. Then $\pi^*\beta_i$ is an effective class, therefore we have a $\sigma$-decomposition
$$\pi^*\beta_i=P_{\sigma}(\pi^*\beta_i)+N_{\sigma}(\pi^*\beta_i),$$
where the positive and negative parts are both effective classes. 
The class $\pi^*\beta_i$ is in $\ker\pi_*$ by the
projection formula. It follows that $P_{\sigma}(\pi^*\beta_i)$ and $N_{\sigma}(\pi^*\beta_i)$ both belong to $\Eff_{n-1}(\pi)$. By \cite[III.1.7.(2) Lemma]{nakayama04}, $P_{\sigma}(\pi^*\beta)$ is the limit of the sequence $P_{\sigma}(\pi^*\beta_i)$, so that $P_{\sigma}(\pi^{*}\beta)$ also belongs to the closed cone $\Mov_{n-1}(\pi)$.
Thus the MSC holds in this case as well.
\end{proof}

\begin{rmk}
A similar argument gives a quick proof of the Strong Conjecture for divisors, using the results of \cite{lehmann11} in place of the arguments of \cite{djv13}.
\end{rmk}

\section{Movable classes: the almost exceptional case} \label{movclasssection}

The main result of the section is Theorem \ref{thrm:almostexceptionalmsc} which proves the Strong Conjecture for movable classes that are ``almost exceptional'', in the sense that their contractibility index (\ref{sec:contractibilityindex}) is one away from the condition for being exceptional.   (Note that the exceptional case is covered by the proof of Theorem \ref{exceptionalclassesthrm} which shows that exceptional nonzero classes are never movable.) Throughout this section we will often work over the complex numbers; this allows us to understand the behavior of birational maps via the following proposition.

\begin{prop}[\cite{fl15} Proposition 3.8] \label{swconjforbirationalsmooth}
Suppose that $\pi: X \to Y$ is a birational morphism of varieties over $\mathbb{C}$ with $Y$ smooth.  Then the kernel of $\pi_{*}: N_{k}(X) \to N_{k}(Y)$ is spanned by classes of effective $k$-cycles contracted by $\pi$.
\end{prop}

We will also need the following properties of basepoint free classes.

\begin{thrm}[\cite{fl13}, Lemma 5.4, Lemma 5.6, Corollary 5.7] \label{thrm:stricttransispullback}
Let $\pi: Y \to X$ be a morphism of projective varieties and let $p: U \to W$ be a basepoint free family of cycles on $X$.
\begin{enumerate}
\item The base change $p_{Y}: U \times_{X} Y \to W$ is also a basepoint free family.  If $X$ is smooth, we have the relation $[F_{p_{Y}}] = \pi^{*}[F_{p}]$.
\item Suppose $X$ is smooth.  For any top-dimensional (effective) cycle $T$ supported on a general fiber $U_{w}$ of $p$, there is a canonical (effective) cycle with support equal to $Y \times_{X} |T|$ whose pushforward to $Y$ represents $\pi^{*}(s|_{T})_{*}[T] \cap Y$. 
\end{enumerate}
In particular, if both $X$ and $Y$ are smooth, $\pi^{*}\bpf^{k}(X) \subset \bpf^{k}(Y)$.  Furthermore, if $Y$ is smooth then the intersection of basepoint free classes on $Y$ is basepoint free.
\end{thrm}

\subsection{Movability and restrictions}

We first identify criteria guaranteeing that the restriction of a movable class to a subvariety is still movable. Throught this section, when $\pi$ is a birational map we will use interchangably the terms ``$\pi$-exceptional'' and ``$\pi$-contracted''.  (Note that these two terms mean the same thing for birational maps by  Example \ref{exmple:exceptionalforbirational}.)

\begin{lem} \label{lem:bpfspanblowup}
Let $\pi: Y \to X$ be a morphism of smooth varieties that is a composition of blow-ups along smooth centers.  Then $N_{k}(Y)$ is spanned by $\pi^{*}N_{k}(X)$ and by a finite set of $\pi$-exceptional effective $k$-cycles, each of which is the pushforward of a basepoint free class on a $\pi$-exceptional divisor.
\end{lem}

\begin{proof}
Using Theorem \ref{thrm:stricttransispullback} inductively, it suffices to consider the case when $\pi$ is the blow-up along a single smooth center $T$.  Since $T$ is smooth, each $N_{j}(T)$ is spanned by basepoint free classes.  Using \cite[Theorem 3.3.(b)]{fulton84} and Proposition \ref{swconjforbirationalsmooth}, we see the kernel of the pushforward map is spanned by the intersection of divisors of a fixed ample divisor class with the pullbacks of these classes.  These intersections represent effective, basepoint free, $\pi$-exceptional cycles by Theorem \ref{thrm:stricttransispullback}.
\end{proof}

\begin{prop} \label{prop:movrestricttoflatfamily}
Let $X$ be a smooth projective variety over $\mathbb{C}$ and let $\alpha$ be the class of a strongly movable family of $k$-cycles $t: R \to S$.  Suppose $U$ is a scheme admitting a flat dominant map $s: U \to X$ and a proper map $q: U \to W$ to an integral variety $W$.  Then for every component $F'$ of a general fiber $F$ of $q$ we have that $\alpha|_{s(F')}$ is movable.

In particular, for any (reduced) component $Z$ of a general member of a basepoint free family on $X$, $\alpha|_{Z}$ is movable on $Z$.
\end{prop}

\begin{proof}
Let $\pi: Y \to X$ be a birational map from a smooth model $Y$ that flattens the map $t: R \to X$.  After possibly passing to a higher model, we may assume that $\pi$ is a compositiion of smooth blow-ups.  Let $\alpha'$ be the class of the strict transform family of $t$ on $Y$; note that $\alpha'$ is a basepoint free class.
By Lemma \ref{lem:bpfspanblowup}, we can write $\alpha' = \alpha + [V]$ where $V$ is a (not necessarily effective) linear combination of $\pi$-exceptional cycles that are general members of basepoint free families on exceptional divisors.

Set $U' := U \times_{X} Y$.  Since the natural morphism $s_{Y}: U' \to Y$ is flat, every component of $U'$ dominates $Y$.  Thus, the $s_{Y}$-image of the general fiber of $p_{Y}: U' \to W$ is the strict transform of the $s$-image of a general fiber of $p$.

Suppose that $Z$ is a $d$-dimensional component of a general fiber of $p$ and $Z'$ is its strict transform on $Y$.  We next verify that:
\begin{enumerate}
\item There is a cycle of class $[V \cdot Z']$ supported on $V \cap Z'$.
\item $V \cap Z'$ is a $\pi$-exceptional cycle.
\end{enumerate}
The two properties together show that $\pi_{*}[V \cdot Z'] = 0$.

Arguing by induction on the number of blow-ups, it suffices to consider the case when $\pi$ is a blow-up of a smooth center $W$, $E$ is the exceptional divisor, and $V$ is a $\pi$-exceptional cycle that is basepoint free in $E$.  To verify (1), note that since $E$ is a Cartier divisor whose support does not contain $Z'$, $[Z']|_{E}$ is represented by a cycle $T$ supported on $Z' \cap E$.  We then apply Theorem \ref{thrm:stricttransispullback} to $T$ and to components of $V$ as basepoint free cycles on $E$.  To verify (2), note that since $s$ and $s_{Y}$ are flat, codimension is preserved upon taking preimages.  In particular, the codimension of $W \cap Z$ in $Z$ is the same as the codimension of $W$ in $X$, and $E \cap Z'$ has codimension $1$ in $Z'$.  Since the fibers of the blow-up of $W$ in $X$ are irreducible, the only way this can happen is if $Z'$ contains every fiber of $\pi$ that it intersects.  Then since $V$ is contracted by $\pi$, $V \cap Z'$ is also contracted by $\pi$.

Thus 
\begin{align*}
\alpha|_{Z} & = (\pi|_{Z'})_{*}(\pi^{*}\alpha|_{Z'}) \\
& = (\pi|_{Z'})_{*}(\alpha'|_{Z'} + [V]|_{Z'}) \\
& = (\pi|_{Z'})_{*}(\alpha'|_{Z'}).
\end{align*}
Since $\alpha'$ is basepoint free, the restriction to $Z'$ is also basepoint free, and hence its pushforward is movable.
\end{proof}

Immediately from Proposition \ref{prop:movrestricttoflatfamily} we obtain:

\begin{cor} \label{cor:movrestricttovgbpf}
Let $X$ be a smooth projective variety over $\mathbb{C}$ and let $\alpha \in \Mov_{k}(X)$.  Suppose $U$ is a scheme admitting a flat dominant map $s: U \to X$ and a proper map $q: U \to W$ to an integral variety $W$.  Then for every component $F'$ of a very general fiber $F$ of $q$ we have that $\alpha|_{s(F')}$ is movable.

In particular, for any (reduced) component $Z$ of a very general member of a basepoint free family on $X$, $\alpha|_{Z}$ is movable on $Z$.
\end{cor}

\subsection{Almost exceptional classes: special case}

We show in Corollary \ref{cor:fiberspecialcase} that the class of a general fiber of a dominant morphism of projective varieties $\pi:X\to Y$ of relative dimension $k$ over $\mathbb C$ is the only class in $\Mov_k(X)$ satisfying $\alpha\cdot\pi^*h=0$ for some (equivalently any) ample divisor class $h$ on $Y$. This will play an important role in the proof of Theorem \ref{thrm:almostexceptionalmsc} in the next subsection.

\begin{lem}\label{lem:controllingsubvarietyalt}Let $X$ be a smooth projective variety of dimension $n$ over $\mathbb C$, and let $\alpha\in \Mov_k(X)$. Then there exists a smooth projective variety $Z$ of dimension $n-k+1$ with a morphism $f:Z\to X$
such that $f_*:N_{n-k}(Z)\to N_{n-k}(X)$ is surjective, and $f^*\alpha\in\Mov_1(X)$.
\end{lem}

\begin{proof}Choose a finite set of $r$ very ample vector bundles $E_i$ on $X$ such that $N^*(X)$ is generated \emph{as a ring} by the Segre classes $s_j(E_i^{\vee})$, and in particular $N^k(X)=N_{n-k}(X)$ is generated as a vector space by monomials of weight $k$ in these Segre classes. In this list repeat each bundle $k$ times (this will make it possible to write any monomial in dual Segre classes of bundles in this set as a monomial in dual Segre classes of different (occurrences) of these bundles in the set). Put $P=\prod_X\mathbb P(E_i)$, and let $\pi:P\to X$ denote the projection map of relative dimension denoted by $p$. Let $\xi_i$ be the pullbacks to $P$ of $c_i(\mathcal O_{\mathbb P(E_i)}(1))$. Using the proof of \cite[Proposition 3.1.(b)]{fulton84}, we see that $N_{n-k}(X)$ is generated by $\pi_*(\prod_{j=1}^{p+k}\xi_{i_j})$, where $i_j$ are arbitrary indices in the set $\{1,\ldots,r\}$ (they may repeat, and the number of occurrences the index $i$ determines which Segre class of $E_i^{\vee}$ appears in the resulting monomial in dual Segre classes of the $E_i$'s).

By Bertini, we can choose smooth representatives $Q_{\underline i}$, one for each $\prod_{j=1}^{p+k}\xi_{i_j}$. We can also ensure that they are disjoint as long as $2(p+k)>n+p$, which can be achieved for example by adding $\mathcal O(A)^{\oplus n+1}$ to the list, for some very ample $A$ on $X$. Put $T_{\underline i}:=\pi_*Q_{\underline i}$. These are cycles on $X$ whose classes generate $N_{n-k}(X)$ by the previous paragraph. 

Let $\widetilde P$ be the blow-up of $P$ along all $Q_i$. Let $Z$ be a general complete intersection of dimension $n-k+1$ on $\widetilde P$. This is smooth, and for each $i$ it contains an $(n-k)$-dimension cycle $\widetilde Q_{\underline i}$ that dominates $Q_{\underline i}$ (because the exceptional divisors are projective bundles over $Q_i$ of relative dimension $p+k-1$, which is the codimension of $Z$). If $f:Z\to X$ denotes the induced map, then $f_*:N_{n-k}(Z)\to N_{n-k}(X)$ is surjective. It remains to show that $f^*\alpha$ is movable.

Let $\sigma:\widetilde P\to P$ be the blow-up map. Since the intersection of an ample class with a movable class is movable (see \cite[Lemma 3.12]{fl14}), and a smooth pullback of a movable class from a smooth variety is movable (see \cite[Lemma 3.6.(2)]{fl14}), it is enough to check that $\sigma^*\pi^*\alpha$ is movable. Let $q_j$ be a sequence of strictly movable families such that $\alpha=\lim_{j\to\infty}[q_j]$. For very general choices of $Q_{\underline i}$, the general member of $\pi^*q_j$ meets $Q_{\underline i}$ properly for all ${\underline i}$ and for all $j$. Then $\sigma^*\pi^*[q_j]$ represents the strict transform in $\widetilde P$ of the strictly movable family $\pi^*q_j$ by \cite[Corollary 6.7.2]{fulton84}. Consequently $\sigma^*\pi^*\alpha$ is still movable.  
\end{proof}

\begin{rmk} \label{rmk:irreduciblebpfgens}
By the first two paragraphs of the argument, any smooth projective variety $X$ over $\mathbb{C}$ admits a basis of $N^{k}(X)$ consisting of classes of basepoint free families whose total space $U$ and general fiber are irreducible.  (Indeed, we can take the $Q_{\underline i}$ to be irreducible by the Bertini theorems, so that the $T_{\underline i}$ are also irreducible. Complete intersections $Q_{\underline i}$ of globally generated line bundles are cycles in basepoint free families on $P$, and then so are their flat pushforwards $T_{\underline i}$ on $X$.)

Moreover, if $f:X\to Y$ is a dominant morphism to another projective variety $Y$, with $\dim Y\geq n-k$, we can also arrange that $\dim f(T_{\underline i})=n-k$ for every $i$. (Indeed we claim that we can choose $T_{\underline i}$ such that $[T_{\underline i}]$ belongs to the interior of $\bpf^k(X)$, hence in particular to the interior of $\Nef^k(X)$. Assuming this, if $\dim f(T_{\underline i})<n-k$, then $[T_{\underline i}]\cdot f^*h^{n-k}=0$ 
for any ample $h\in N^1(Y)$, which implies $f^*h^{n-k}=0$, leading to the contradiction $\dim Y<n-k$. For the claim, replace first each $E_{i}$ by $E_{i}\otimes\det E_i$, and add the ample line bundles $\det E_i$ to the initial list. These have
the same linear span of dual Segre monomials. Since complete intersections belong to the interior of $\bpf^r(X)$ for all $r$, and dual Segre monomials of globally generated bundles are basepoint free, by Theorem \ref{thrm:stricttransispullback} it is enough to check that $s_j((E_i\otimes\det E_i)^{\vee})$ belongs to the interior of $\bpf^j(X)$. By \cite[Example 3.1.1]{fulton84}, the class $s_j((E_i\otimes\det E_i)^{\vee})$ is a positive combination of classes in $\bpf^j(X)$, one of which is a positive multiple of the interior class $c_1^j(E_i)$.)\qed
\end{rmk}

\begin{cor} \label{cor:fiberspecialcase}
Let $\pi: X \to Y$ be a surjective morphism of projective varieties over $\mathbb{C}$ with relative dimension $k$.  Let $\alpha \in \Mov_{k}(X)$ be such that $\alpha \cdot \pi^{*}H = 0$ for an ample divisor $H$ on $Y$.  Then $\alpha$ is proportional to the class of a fiber.  
\end{cor}

\begin{proof}
First suppose $X$ is smooth.  Choose $Z$ as in Lemma \ref{lem:controllingsubvarietyalt} with morphism $f:Z\to X$.  Note that $\pi\circ f: Z \to Y$ has relative dimension $1$.  Then $f^*\alpha$ is a movable curve class that pushes forward to $0$ on $Y$.  Thus it is proportional to the class of a fiber by Corollary \ref{cor:msccurvesreldim1}.  So for any divisor $D$ on $Z$ we have $$D \cdot \alpha = c\deg(\pi\circ f|_{D}).$$ But since $f_{*}:N_{n-k}(Z) \to N_{n-k}(X)$ is surjective, the same proportionality relationship holds for $(n-k)$-cycles on $X$.  So $\alpha$ is proportional to the class of a general fiber of the map $\pi$.

When $X$ is singular, let $\phi: X' \to X$ be a smooth birational model and let $\alpha'$ be a movable preimage of $\alpha$. By Theorem \ref{nullmovpush}, $\alpha'\cdot \phi^*\pi^*H=0$.  Applying the smooth case to $\alpha'$, we see that $\alpha'$ is proportional to the class of a general fiber of $\pi \circ \phi$.  Pushing forward, we see that $\alpha$ is also proportional to the class of a general fiber of $\pi$.
\end{proof}

\subsection{Almost exceptional classes: general case}

In Theorem \ref{thrm:almostexceptionalmsc} and its corollary we prove the almost exceptional case of the Movable Strong Conjecture over $\mathbb C$, and discuss how this settles most of the cases of the Strong Conjecture for morphisms from complex fourfolds. 

\begin{lem} \label{lem:finitepullback}
Let $g: X \to Y$ be a finite dominant map of projective varieties with $Y$ smooth.  Let $\alpha \in N^{n-k}(X)$.  Suppose that there is a finite collection of $(n-k)$-dimensional subvarieties $\{ W_{i} \}$ of $Y$ containing general points of $Y$, such that if $Z_{1}$ and $Z_{2}$ are $(n-k)$-dimensional integral subvarieties of $X$ both mapping to the same $W_{i}$, then $$\frac{\alpha{\cdotp}Z_{1}}{\deg(Z_{1}/W_{i})} = \frac{\alpha{\cdotp}Z_{2}}{\deg(Z_{2}/W_{i})}.$$  Then there is some $\beta \in N^{n-k}(Y)$ so that for any $Z$ above one of the $W_{i}$ we have $\alpha \cdot Z = g^{*}{\beta} \cdot Z$.   If $\alpha \cap [X]$ is movable, we may ensure that $\beta \cap [Y]$ is also movable.
\end{lem}

\begin{proof}
Let $d = \deg(X/Y)$.  Set $\beta = \frac{1}{d}{\cdotp}g_{*}(\alpha \cap [X])$; since $Y$ is smooth we can think of $\beta \in N^{n-k}(Y)$.  If $E$ is an irreducible subvariety with $g(E) = W_{i}$, we have
\begin{align*}
g^{*}\beta \cdot E & = \beta \cdot g_{*}E = \frac{1}{d}\cdot g_{*}(\alpha \cap [X]) \cdot \deg(E/W_{i})W_{i} \\
& = \frac{\deg(E/W_{i})}{d} \alpha \cdot (g^{*}[W_{i}] \cap [X]).
\end{align*}
We check that $g^*[W_i]\cap[X]=[g^{-1}W_i]$. It is enough to check the equality in Chow groups. Note that $g^{-1}W_i$ has the expected dimension because $Y$ is smooth and $g$ is finite. Using the restriction sequence for Chow groups, the finiteness of $g$, and the generality assumption on $W_i$, it is enough to check the equality of Chow classes over the flat locus of $g$. Over the flat locus the equality is the definition of the flat pullback of $W_i$ via the compatibility between flat and smooth pullback (see \cite[Proposition 8.1.2.(a)]{fulton84}).

Using this equality, we find
\begin{align*}
g^{*}\beta \cdot E  & = \frac{\deg(E/W_{i})}{d} \alpha \cdot \sum_{g(E_{j}) = W_{i}} \mathrm{ramdeg}(E_{j}/W_{i}) E_{j} \\
& = \frac{\deg(E/W_{i})}{d} \cdot \sum_{g(E_{j}) = W_{i}} \mathrm{ramdeg}(E_{j}/W_{i})\cdot\frac{\deg(E_{j}/W_{i})}{\deg(E/W_{i})}{\cdotp} \alpha \cdot E.
\end{align*}
Since the $W_{i}$ contain general points, the smooth locus of any component of the preimage intersects the smooth locus of $X$, so that by \cite[Example 4.3.7]{fulton84}
\begin{equation*}
d =  \sum_{g(E_{j}) = W_{i}} \mathrm{ramdeg}(E_{j}/W_{i})\cdot \deg(E_{j}/W_{i}).
\end{equation*}
This implies that $g^{*}\beta \cdot E = \alpha \cdot E$.  The final statement follows since $\beta$ is proportional to the pushforward of $\alpha$ and movability is preserved by pushforward.
\end{proof}

We also need a version of Lemma \ref{lem:finitepullback} where the morphism $\pi$ is allowed to be generically finite.  The cost is that we need stronger positivity assumptions on the cycles involved.

\begin{lem} \label{lem:genfinitepullback}
Let $\pi: X \to Y$ be a generically finite dominant map of smooth projective varieties over $\mathbb{C}$.  Let $\alpha \in N_{k}(X)$.  Let $T_{1}, \ldots, T_{r}$ be $(n-k)$-cycles on $Y$ which are components of general members of bpf families.  Suppose that for each $T_{i}$ there is a constant $s_{i}$ so that
\begin{equation*}
\alpha \cdot Z =  s_{i} \deg(Z/T_{i})
\end{equation*}
for any subvariety $Z$ lying above $T_{i}$.  Then there is a class $\beta \in N^{n-k}(Y)$ such that $\alpha \cdot Z = \pi^{*}\beta \cdot Z$ for any $Z$ lying above some $T_{i}$.

If $\alpha$ is movable, then so is $\beta \cap [Y]$.
\end{lem}

\begin{proof}
Let $\pi': X' \to Y'$ be a flattening via a birational map $f_{Y}: Y' \to Y$ with $Y'$ smooth.  Let $f_{X}: X' \to X$ denote the corresponding birational map.  For each $T_{i}$ denote its strict transform on $Y'$ by $T_{i}'$. Since $T_i$ is a component of a general member of a bpf family, $[T'_i]=f_Y^*[T_i]$ by Theorem \ref{thrm:stricttransispullback}. Note that any $(n-k)$-dimensional subvariety $Z$ on $X$ dominating some $T_{i}$ is again a component of a basepoint free family, since it is a component of the base-change of $p$ (see \cite[Lemma 5.6]{fl13}). Thus, the pullback $f_{X}^{*}[Z]$ coincides with the class of its strict transform, so that $$\pi'_{*}f_{X}^{*}[Z] = \deg(Z/T_{i}) [T_{i}']=f_Y^*\pi_*[Z].$$
Consider the pullback $f_{X}^{*}\alpha$.  It still satisfies the intersection compatibility with degree for cycles lying over the $T_{i}'$.  So Lemma \ref{lem:finitepullback} shows that $\beta' := \frac{1}{\deg\pi}\pi'_{*}(f_{X}^{*}\alpha)$ has the property that
\begin{equation*}
\pi'^{*}\beta' \cdot [Z'] = f_{X}^{*}\alpha \cdot [Z']
\end{equation*}
for any cycle $Z'$ lying above one of the $T_{i}'$.
Define $\beta = f_{Y*}\beta'$.  Then for any $Z$ lying above one of the $T_{i}$
\begin{align*}
\pi^{*}\beta \cdot [Z] & = f_{Y*}\beta' \cdot \pi_{*}[Z] \\
& = \beta' \cdot \deg(Z/T_{i}) [T_{i}'] \\
& = \beta' \cdot \pi'_{*}f_{X}^{*}[Z] \\
& = \alpha \cdot [Z].
\end{align*}
As for the final statement, we see that $\beta \cap [Y]$ is movable since it is constructed to be proportional to the pushforward of $\alpha$.
\end{proof}

Finally we prove the main theorem of this section.

\begin{thrm} \label{thrm:almostexceptionalmsc}
Let $\pi: X \to Y$ be a surjective morphism with connected fibers of smooth projective varieties  over $\mathbb{C}$ of relative dimension $e$.  Suppose $\alpha \in \Mov_{k}(X)$ for some $k \geq e$ and that $\alpha \cdot \pi^{*}H^{k-e+1} = 0$ for some ample divisor $H$ on $Y$.  Then there is a diagram
$$\xymatrix{
X' \ar[d]^{\pi'}\ar[r]^(.55){f_{X}} & X \ar[d]^{\pi} \\
Y'  \ar[r]^(.55){f_{Y}} & Y
}$$
with $f_{X}$ and $f_{Y}$ birational, $Y'$ smooth, and $\pi'$ flat and a class $\beta \in \Mov_{k-e}(Y')$ such that $f_{X*}\pi'^{*}\beta = \alpha$.
\end{thrm}

Note that the classes satisfying the condition $\alpha \cdot \pi^{*}H^{k-e+1}=0$ are almost exceptional: the contractibility index (cf. \S\ref{sec:contractibilityindex}) is (at most) one away from the condition for being exceptional.

\begin{proof} Let $n$ be the dimension of $X$ and $d$ the dimension of $Y$ so that $e = n-d$. Consider a flattening $\pi': X' \to Y'$ of $\pi$ with $Y'$ smooth:
$$\xymatrix{
X' \ar[d]^{\pi'}\ar[r]^(.55){f_{X}} & X \ar[d]^{\pi} \\
Y'  \ar[r]^(.55){f_{Y}} & Y
}$$
Let $\psi: \widetilde{X} \to X'$ be a resolution and let $\rho: \widetilde{X} \to Y'$ denote the composition of $\psi$ and $\pi'$.  Let $\tilde\alpha$ be a movable preimage of $\alpha$ on $\widetilde{X}$. From Theorem \ref{nullmovpush} we see that 
\begin{equation*}
\tilde\alpha\cdot (\rho^{*}f_{Y}^{*}H)^{k-e+1} = 0.
\end{equation*}
Writing $f_{Y}^{*}H = A + E$  for an effective Cartier divisor $E$ and an ample divisor $A$, and using the movability of $\tilde\alpha$ and Lemma \ref{dercycle}, we also obtain
\begin{equation*}
\tilde\alpha \cdot \rho^{*}A^{k-e+1} = 0.
\end{equation*}
Thus $\tilde\alpha$ is almost exceptional for the map $\rho$.  By pushing forward, we observe $(\psi_*\tilde\alpha)\cdot\pi'^*A^{k-e+1}=0$, so that $\psi_{*}\tilde\alpha$ is almost exceptional for the map $\pi'$.

Let $\{ p_{i}: U_{i} \to W_{i} \}$ be a finite collection of basepoint free families whose classes span $N^{k}(X)$. We can choose them such that the $U_i$'s and the general fiber of each $p_i$ are irreducible and their images on $X$ are not contracted by $\pi$ (see Remark \ref{rmk:irreduciblebpfgens}).  We will do a series of constructions to $p_{i}$; at each step, we will replace $W_{i}$ by an open subset which for simplicity we also denote by $W_{i}$.  The strict transform families $p_{i}': U_{i}' \to W_{i}$ on $X'$ are still basepoint free by Theorem \ref{thrm:stricttransispullback} and the images on $X'$ of general fibers are not contracted by $\pi'$.  Consider the diagram
$$\xymatrix{
U_{i}' \times_{Y'} X' \ar@/_.3in/[dd]_{q_{i}}\ar[d]^{\pi'_i}\ar[r]^(.55){t_{i}} & X' \ar[d]^{\pi'} \\
U_{i}' \ar[d]^{p_{i}'}\ar[r]^(.55){\pi' \circ s_{i}'} & Y'  \\
W_{i}}
$$
Note that the map $t_{i}$ is flat and the map $q_{i}$ is proper, making $U_{i}' \times_{Y'} X'$ a basepoint free family on $X'$.  We then take the pushout diagram to $\widetilde{X}$ to obtain a basepoint free family $r_{i}: \widetilde{U}_{i} \to W_{i}$ with a flat map $\widetilde{s}_{i}: \widetilde{U}_{i} \to \widetilde{X}$.  For each $i$, let $F_{i}$ be a very general fiber of $r_{i}$. Since $\pi'$ is flat equidimensional with
irreducible general fiber, Theorem \ref{thrm:stricttransispullback} implies that $F_i$ is irreducible. Since the general cycles in $p_i$ do not contract in $Y$, it follows that $\dim\widetilde{s}_i(F_i)=n-k+e$ and $\dim\rho(\widetilde{s}_i(F_i))=n-k$.

We claim that there is some $\beta \in \Mov_{k-e}(Y')$ such that $\tilde\alpha \cdot Z = \rho^{*}\beta \cdot Z$ for every $(n-k)$-cycle $Z$ contained in some $\widetilde{s}_{i}(F_{i})$.  This will conclude the proof of the theorem: since each $\widetilde{s}_i(F_{i})$ contains the strict transform of the corresponding cycle of $p_{i}$ and since the classes of these strict transforms are the pullbacks of a basis of $N^{k}(X)$, we see that $f_{X*}\pi'^{*}\beta = f_{X*}\psi_{*}\tilde\alpha = \alpha$.



With $V=\widetilde{s}_{i}(F_{i})$, we have $\dim V=(n-k+e)$, and $\dim\rho(V)=n-k$.   By Corollary \ref{cor:movrestricttovgbpf}, we see that $\tilde\alpha|_{V}$ is a movable class in $N_{e}(V)$.  Setting $r = k-e + 1$, with $\imath: V \hookrightarrow \widetilde{X}$ denoting the natural map, for $C\gg0$, we have
\begin{align*}
\imath_*\imath^*(\tilde\alpha\cdot\rho^{*}A) & = \tilde\alpha \cdot [V] \cdot \rho^{*}A^{r-(k-e)}  \\
& \preceq \tilde\alpha \cdot C \rho^{*}A^{r} = 0,
\end{align*}
so that $\tilde\alpha|_{V}$ satisfies the conditions for Corollary \ref{cor:fiberspecialcase}.  The conclusion is that $\tilde\alpha|_{V}$ is proportional to the class of a fiber of $\rho|_{V}$.  Thus $\tilde\alpha \cdot [V] \in N_{e}(X)$ is proportional to the class of a fiber of $\rho$ via some constant $b$.  


As we vary $i$, the argument above yields constants $b_{i}$ which are a priori unrelated.  Let $\{ T_{i} \}$ be the $(n-k)$-dimensional subvarieties of $Y'$ that are the images of the $V_{i}$.  Note that each $T_{i}$ is the image of a general member of a basepoint free family on $Y'$ (constructed as the flat image of the $p_{i}'$).

 Let $M$ be a very general complete intersection of ample divisors on $\widetilde{X}$ of dimension $d$.  By very generality, $\tilde\alpha|_{M}$ is movable.  The previous paragraph shows that for each $T_{i}$, intersections of $\tilde\alpha|_{M}$ against cycles with support contained in each $\rho|_{M}^{-1}(T_{i})$ are proportional via some constant $b_{i}$.  Lemma \ref{lem:genfinitepullback} shows that there is a movable class $\beta$ on $Y'$ such that $\rho^{*}\beta|_{M} \cdot Z = \tilde\alpha |_{M} \cdot Z$ for any cycle $Z$ lying above one of the $T_{i}$.  But since intersections are compatible against degree for \emph{any} subvariety of the $V_{i}$, we see that $\tilde\alpha \cdot Z = \rho^{*}\beta \cdot Z$ for any cycle $Z$ with support contained in $\widetilde{s}_{i}(F_{i})$ which has dimension $(n+e-k)$.\end{proof}

\begin{cor}\label{cor:mscalmostexc}
Let $\pi: X \to Y$ be a surjective morphism of projective varieties over $\mathbb{C}$ of relative dimension $e$.  Suppose $\alpha \in \Mov_{k}(X)$ for some $k \geq e$ and that $\alpha \cdot \pi^{*}H^{k-e+1}=0$ for some ample divisor $H$ on $Y$.  Then the MSC holds for $\alpha$. In particular the MSC is true when $e=1$.
\end{cor}


\begin{proof}
Since we know the MSC for generically finite maps, arguing as in Remark \ref{bothsmconnC} we may assume that $X$ and $Y$ are smooth and $\pi$ has connected fibers.  Applying Theorem \ref{thrm:almostexceptionalmsc} we obtain a smooth birational model $Y'$ and a class $\beta \in \Mov_{k-e}(Y')$.  Let $\{ Z_{i} \}_{i=1}^{\infty}$ be a sequence of strictly movable cycles whose classes limit to $\beta$.  Since $\pi'$ is flat, each $\pi'^{*}[Z_{i}] = [\pi'^{-1}Z_{i}]$ is the class of an effective $\pi'$-contracted cycle which is strictly movable by \cite[Lemma 3.6]{fl14}.  The image of each $\pi'^*Z_i$ under $(f_{X})_*$ is a movable $\pi$-contracted cycle, and the corresponding classes limit to $\alpha$.
\end{proof}

In fact, we can weaken the hypotheses of Theorem \ref{thrm:almostexceptionalmsc} without changing the proof.  We now explain this stronger version.

\begin{rmk}\label{rmk:weakext}We have not used the full strength of the movability condition on $\alpha$
in this section. In Lemmas \ref{lem:finitepullback} and \ref{lem:genfinitepullback}, one can replace movability with any notion invariant under pushforward by surjective morphisms. Theorem \ref{thrm:almostexceptionalmsc} uses three properties of movability:

\begin{enumerate}[i)]
\item Movable classes admit movable preimages by surjective maps.
\item Corollary \ref{cor:movrestricttovgbpf}.
\item Movable curves are movable in the sense of \cite{bdpp04}.
\end{enumerate}
\end{rmk}

Consider the following partial substitute:
\begin{defn}Say that a class $\alpha\in\Eff_k(X)$ is \emph{weakly movable} if there exists a sequence $\{V_i\}_{i=1}^{\infty}$ of effective $k$-cycles on $X$ such that 
$\alpha=\lim_i[V_i]$, and for each (reducible) divisor $E$ on $X$, infinitely many of the $V_i$'s meet $|E|$ properly.
\end{defn}

The following remark explains how the analogue of Theorem \ref{thrm:almostexceptionalmsc} for weakly movable classes is proved.

\begin{rmk}\label{rmk:weakprop}Weakly movable classes are invariant under pushforward by surjective morphisms and have the following additional characteristics:
\begin{enumerate}[i)]
\item Extremal classes in $\Eff_k(X)$ that are not the pushforward of any pseudoeffective class on a (reducible) divisor on $X$ are weakly movable and admit weakly movable preimages by surjective \emph{generically finite} maps.
\item The analogue of Corollary \ref{cor:movrestricttovgbpf} holds for weakly movable classes when the basepoint free family has irreducible general fiber that is not contracted by the map to $X$, i.e. it produces a nonzero class.
\item Weakly movable curves are movable in the sense of \cite{bdpp04}.
\end{enumerate}
(If $\pi:X\to Y$ is surjective, $\alpha\in\Eff_k(X)$ is weakly movable, and $\{V_i\}$ is a sequence of cycles that verifies its movability, then $\{\pi_*V_i\}$ is a sequence of cycles that verifies the weak movability of $\pi_*\alpha$. 

If $\alpha\in\Eff_k(Y)$ is extremal and not pushed from a divisor on $Y$, then there exists a sequence of cycles $V_i$ on $Y$ having irreducible support, and such that every subsequence is dense in $Y$. The sequence $\{V_i\}$ verifies the weak movability of $\alpha$. Furthermore, any extremal pseudoeffective preimage $\beta\in\Eff_k(X)$ with $\pi_*\beta=\alpha$ is likewise not pushed from a divisor on $X$. Therefore $\beta$ is also weakly movable and this proves i).

The justification of ii) is a standard relative Hilbert scheme argument. Let $\alpha$ be weakly movable and let $V_i$ be $k$-cycles that verify its weak movability. Let $p:U\to W$ be a projective morphism with irreducible general fiber to $W$ integral and let $s:U\to X$ be an equidimensional flat morphism. For very general $w\in W$, the fiber $U_w$ sits in general position relative to all $V_i$'s. Consider the relative Hilbert scheme $\mathcal H_j$ parameterizing pairs $(w,D_w)$, where $D_w$ is a divisor on $U_w$ that contains a component of $s^{-1}V_i\cap U_w$ for all $i\geq j$. If the weak movability of restrictions (computed as proper intersections in the sense of \cite[\S7]{fulton84}) fails, then by the uncountability of the base field, some component of some $\mathcal H_j$ dominates $W$. One then uses the universal family over an appropriate subvariety of this component to construct a divisor whose image in $X$ is a divisor that meets all but finitely many of the $V_i$'s improperly.

It is an immediate consequence of the definition that a weakly movable curve class has nonnegative intersection with any effective Cartier divisor. Then iii) follows by \cite{bdpp04}.) 
\end{rmk}

\begin{cor} \label{cor:scfourfolds}
The Strong Conjecture holds for surjective morphisms from fourfolds to threefolds over $\mathbb C$. More generally, it holds for almost exceptional classes on fourfolds regardless of the target.\end{cor}

\begin{proof}
Let $\pi:X\to Y$ be a morphism from a fourfold, and let $\alpha\in\Eff_k(X)$ satisfy $\pi_*\alpha=0$. It is enough to treat the case of surface classes ($k=2$), since curves and divisors are covered by \cite[Theorem 1.4]{djv13}. By Remark \ref{mindim} and the reduction steps of \cite{djv13} we may assume that $\pi$ is surjective and furthermore that its relative dimension is $e\in\{0,1,2\}$. If $e=0$, i.e. $\pi$ is generically finite, then $\alpha$ is exceptional, and hence pushed forward from a subscheme of $X$ by Theorem \ref{exceptionalclassesthrm}.   We obtain the Strong Conjecture for $\alpha$ from the Strong Conjecture for threefolds. 

When $e=1$, the condition $\pi_*\alpha=0$ is equivalent to $\alpha\cdot\pi^*h^{k-e+1}=\alpha\cdot\pi^*h^2=0$ for some $h$ ample on $Y$. In particular $\alpha$ is almost exceptional. We may assume that $\alpha$ is extremal and is not pushed from any divisor on $X$, otherwise we reduce to the case when $\alpha$ is a divisor class on a threefold.
Then $\alpha$ is weakly movable and Remarks \ref{rmk:weakext} and \ref{rmk:weakprop} show that the proof of Theorem \ref{thrm:almostexceptionalmsc} carries through. 

The same argument works when $e=2$ if $\alpha$ is almost exceptional. 
\end{proof}

\begin{rmk}The only unsettled case of the Strong Conjecture in dimension $4$ over $\mathbb C$ is that of a surjective morphism $\pi:X\to Y$ to a surface and of classes $\alpha\in\Eff_2(X)$ with $\alpha\cdot\pi^*h^2=0$, but $\alpha\cdot\pi^*h\neq 0$, where $h$ is an ample divisor class on $Y$.
\end{rmk}

\begin{ques} Are weakly movable classes movable?
\end{ques}

\noindent As mentioned above, this is true for curves by \cite{bdpp04} (which holds in arbitrary characteristic; see \cite[Section 2.2]{fl14}).  It is also true for divisors: using Remark \ref{rmk:weakprop}.(ii) one reduces to the case of smooth varieties.  But on smooth varieties the weakly movable condition for $L$ implies that $N_{\sigma}(L) = 0$ so that $L$ is movable (see \cite{nakayama04} and \cite{mustata13}).  

\nocite{*}
\bibliographystyle{amsalpha}
\bibliography{swshort}

\end{document}